\documentclass[10pt]{article}
\usepackage[english]{babel}
\usepackage{amsfonts}
\usepackage[utf8]{inputenc}
\usepackage{amsthm}
\usepackage{ mathrsfs }
\usepackage{ amsmath }
\usepackage{mathrsfs}
\usepackage{enumerate}
\usepackage{enumitem}
\usepackage{cite}
\usepackage{makeidx}
\usepackage{color}
\usepackage{graphicx}
\usepackage[makeroom]{cancel}
\usepackage[normalem]{ulem}
\usepackage{ amssymb }
\usepackage[a4paper, total={6in, 9in}]{geometry}
\usepackage{hyperref}
\usepackage{mathtools}
\usepackage{esint}

\author{Nicolò De Ponti\thanks{Dipartimento di Matematica e Applicazioni, Universit\`a di Milano--Bicocca,\\email: nicolo.deponti(at)unimib.it}\,, Sara Farinelli\thanks{Lagrange Mathematics and Computation Research Center,\\email: sfarinel(at)sissa.it}\,, Ivan Yuri Violo\thanks{Centro di Ricerca Matematica Ennio De Giorgi, Scuola Normale Superiore di Pisa,\\email: ivan.violo(at)sns.it}}
\newcommand{\eps}{\varepsilon}

\newcommand{\rr}{\mathbb{R}}
\newcommand{\nn}{\mathbb{N}}

\newcommand{\nchi}{{\raise.3ex\hbox{$\chi$}}}

\newcommand{\sfd}{{\sf d}}

\renewcommand{\phi}{\varphi}
\newcommand{\restr}[1]{\lower3pt\hbox{$|_{#1}$}}

\newcommand{\X}{{\rm X}}

\newcommand{\fr}{\penalty-20\null\hfill$\blacksquare$} 
\definecolor{mygray}{gray}{0.9}
\newcommand{\diam}{\text{diam}}

\newcommand{\mea}{\mathfrak{m}}
\newcommand{\mm}{\mathfrak{m}}
\newcommand{\Per}{{\rm{Per}}}
\newcommand{\LIP}{\mathsf{LIP}}

\renewcommand{\d}{{\mathrm d}}
\newcommand{\loc}{\mathsf{loc}}
\newcommand{\W}{\mathit{W}^{1,2}}

\newcommand{\supp}{\mathop{\rm supp}\nolimits} 

\newcommand{\D}{{\sf D}}

\newcommand{\lip}{{\rm lip}}

\newcommand{\Xdm}{(\X,\sfd,\mm)}
\newcommand{\Leb}{{\mathscr L}}      
\newcommand{\RCD}{\mathrm{RCD}}

\newcommand{\R}{\mathbb{R}}

\newcommand{\rcd}{\mathrm{RCD}}

\newcommand{\N}{\mathbb{N}}

\newcommand{\cN}{{\mathcal N}}

\newcommand{\dom}{{\sf D}}

\newcommand{\Y}{{\rm Y}}

\newcommand{\cD}{{\ensuremath{\mathcal D}}}

\newcommand{\rmCh}{{\sf Ch}}

\renewcommand{\cal}{\mathcal}
\renewcommand{\limsup}{\varlimsup}
\renewcommand{\liminf}{\varliminf}

\newcommand{\essb}{\partial^e }
\newcommand{\hau}{\mathcal H}

\newcommand{\sfdomega}{\sfd \restr{\overline{\Omega}}}
\newcommand{\mmomega}{\mm \restr{\overline{\Omega}}}
\theoremstyle{plain}
\newtheorem{theorem}{Theorem}[section]
\newtheorem{lemma}[theorem]{Lemma}
\newtheorem{prop}[theorem]{Proposition}
\newtheorem{proposition}[theorem]{Proposition}
\newtheorem{cor}[theorem]{Corollary}

\theoremstyle{definition}
\newtheorem{definition}[theorem]{Definition}
\newtheorem{remark}[theorem]{Remark}

\numberwithin{equation}{section}

\title{Pleijel nodal domain theorem in non-smooth setting}

\date{}
\begin{document}
\maketitle

\bigskip

\begin{abstract}
	We prove  the Pleijel theorem in  non-collapsed RCD spaces, providing an asymptotic upper bound on the number of nodal domains of Laplacian eigenfunctions. As a consequence, we obtain that the Courant nodal domain theorem holds except at most for a finite number of eigenvalues. More in general, we show that the same result is valid for {Neumann (resp. Dirichlet)} eigenfunctions on uniform domains (resp. bounded open sets). This is new even in the Euclidean space, where the Pleijel theorem in the Neumann case was open under low boundary-regularity.
\end{abstract}		
 
	\vskip .9cm

\setcounter{tocdepth}{1}
\tableofcontents

\section{Introduction}
Given a continuous eigenfunction $u$ of a linear operator $L$, there is a lot of interest in studying the properties of its nodal domains, the latter being defined as the connected components of the set $\{u\neq 0\}$. In the usual setting $L$ is an operator of differential nature, with discrete spectrum $\lambda_1\le \lambda_2\le \ldots\le \lambda_k\le\ldots,$ and a classical problem is to bound the number of nodal domains of $u_k$ in terms of $k$. Here $u_k$ is an eigenfunction of eigenvalue $\lambda_k$.

There are two main results known in this direction. The first one, due to Courant \cite{Co23} (see also \cite{CoHi53}), provides a pointwise bound: \emph{for every} $k$ the number of nodal domains of $u_k$ is less than or equal to $k$. The second one is due to Pleijel \cite{Ple56} and provides an asymptotic upper bound, which implies that \emph{for sufficiently large} $k$ the number of nodal domains of $u_k$ is strictly less than $k$.

The theorems of Courant and Pleijel have been deeply investigated in a different number of situations, see \cite{AHH,BBGRS,BeHe18,BM82,Bou15,Char18,CDDG,DGL01,Donn14,HaSh23,KS20,Le19,Pol09,Stein14} for a non-exhaustive list. The main goal of the present paper is to obtain an asymptotic upper bound on the number of nodal domains for Dirichlet and Neumann Laplacian eigenfunctions in the setting of possibly non-smooth metric measure spaces.
In particular, we focus on the class of $\rcd(K,N)$ spaces  consisting of metric measure spaces satisfying a synthetic notion of Ricci curvature bounded from below by $K$ and dimension bounded from above by $N$ (see \cite{AmbICM} and Section \ref{sec:rcd} for more details). Nevertheless, our analysis is of interest already in the Euclidean case since we prove the Pleijel theorem for Lipschitz and even more rough domains (see Corollary \ref{cor: mainEucl} below and the subsequent discussion). Indeed, the validity of a Pleijel result in the Neumann case with boundary regularity below $C^{1,1}$ was an open question in the field  (see the comments after Remark 1.2 in \cite{HaSh23}).

About the Courant nodal domain theorem, let us just briefly mention that its validity is open for $\rcd$ spaces. This is mainly due to the fact that the \emph{weak unique continuation property} for the Laplacian in this setting is currently not known. We refer to \cite{DZ23,DZsec} for more on this problem, where also the failure of the \emph{strong} unique continuation property in the $\rcd$ setting is shown. We remark that a worse, but still pointwise, upper bound on the number of nodal domains can be easily deduced from the variational characterization of the eigenvalues (see \cite{KS20}).

\vspace{3mm}

Before stating our main result, let us first introduce the setting and some notations referring to Section \ref{sec:pre} for the precise definitions. Our investigation deals with eigenfunctions of the Dirichlet or Neumann Laplacian $\Delta_\cD$, $\Delta_\cN$  in a bounded domain $\Omega\subset \X$ in a  $\rcd(K,N)$ space $(\X,\sfd,\hau^N)$,  where $\hau^N$ denotes the $N$-dimensional Hausdorff measure in $(\X,\sfd)$. As usual in this kind of problems, some additional assumptions are required to deal with the Neumann case and we will demand that $\Omega$ is a \textit{uniform domain} (see Definition \ref{def: unif dom}). We will clarify below why we need to restrict our attention to RCD spaces endowed with the Hausdorff measure, called \textit{non-collapsed} in the literature, instead of considering the full RCD class.   Here we limit ourselves to mention that these assumptions are sufficient for the Dirichlet and Neumann Laplacian in $\Omega$ to have discrete spectrum and for the eigenfunctions to be continuous. We list the Dirichlet and Neumann eigenvalues respectively by
\[
\begin{split}
    &0\leq\lambda_1^\cD(\Omega)\le \lambda_2^\cD(\Omega)\le \dots \le\lambda_k^\cD(\Omega)\le \dots \rightarrow +\infty,\\
    &0=\lambda_1^\cN(\Omega)\le \lambda_2^\cN(\Omega)\le \ldots\le \lambda_k^\cN(\Omega)\le\ldots\rightarrow +\infty,
\end{split}
\]counted with multiplicity. Thanks to the continuity of a Laplacian eigenfunction $u$ in our setting, it makes sense to define its \emph{nodal domains}, which are the connected components of $\Omega \setminus \{u=0\}$.
For any $k\in \N$ we can now define
$$M_\Omega^\cD(k):=\sup\left\{\# \text{ of nodal domains of } u : u \ \textrm{Dirichlet eigenfunction of eigenvalue } \ \lambda_k^\cD(\Omega)\right\}$$
and analogously $M_\Omega^\cN$ in the Neumann case (see Def. \ref{def:counting} for more detailed definitions of  $M_\Omega^\cD,M_\Omega^\cN$).
We finally denote by $j_{\alpha}$ the first positive zero of the Bessel function of index $\alpha>0$ and by $\omega_N$ the volume of the unit ball in the $N$-dimensional Euclidean space.

\begin{theorem}[Pleijel theorem in RCD setting - Neumann and Dirichlet cases]\label{main th}
Let $(\X,\sfd,\hau^N)$ be an $\rcd(K,N)$ space, with $K\in \R$ and $N\ge 2$, and let $\Omega\subset \X$ be an open and bounded set. Then
\begin{equation}\label{eq: main resultDir}
\limsup_{k\to +\infty} \frac{M_\Omega^\cD(k)}{k}\le \frac{(2\pi)^N}{\omega^2_N j^N_{\frac{(N-2)}{N}}}<1\, .
\end{equation}
If moreover $\Omega\subset \X$ is a uniform domain, then
\begin{equation}\label{eq: main resultNeum}
\limsup_{k\to +\infty} \frac{M_\Omega^\cN(k)}{k} \le \frac{(2\pi)^N}{\omega^2_N j^N_{\frac{(N-2)}{N}}}<1\, .
\end{equation}
In particular, for every $k \in \nn$ large enough every Dirichlet (resp. Neumann) eigenfunction of eigenvalue $\lambda_k^\cD(\Omega)$ (resp.\ $\lambda_k^\cN(\Omega)$) in any $\Omega$ bounded open set (resp.\ uniform domain) has less than $k$ nodal domains.
\end{theorem}

There has been recently a growing interest in the study of  eigenvalues and eigenfunctions of the Laplacian and their zero set in the setting of RCD spaces (see  \cite{AHPT21,AH18,H23,HM21,DePF22,BF22,CF21,DZ23,ZZ,AHT,DZsec}). However, to the best of our knowledge,  Theorem \ref{main th}  is the first non-trivial result related to nodal domains.

The class of non-collapsed $\rcd(K,N)$ space includes non-collapsed Ricci limit spaces \cite{CC1,CC2}  and finite dimensional Alexandrov spaces \cite{Pet,ZZ10}, and our result is new also for these classes of spaces where the Courant's nodal domain theorem is not known. 
Additionally, thanks to the recent \cite{Raj20}, we know that every $\rcd(K,N)$ space contains a rich class of non-trivial uniform domains, hence it is possible to find many sets that satisfy the assumptions of our result also in the Neumann case.
We recall that, in the somewhat easier Dirichlet case, Theorem \ref{main th} goes back to the work of Pleijel \cite{Ple56} in the Euclidean plane and to B\'{e}rard and Meyer \cite{BM82} for smooth Riemannian manifolds.

When $\X$ is bounded it is  allowed to take $\Omega=\X$ in the Theorem \ref{main th}. In this case Neumann eigenfunctions coincide with the usual Laplacian eigenfunctions on $\X$ and we have the following.
\begin{cor}\label{cor: mainRCD}
Let $(\X,\sfd,\hau^N)$ be a compact $\rcd(K,N)$ space, with $K\in \R$ and $N\ge 2$. Denote by $\{\lambda_k\}_{k \in \nn}$ the eigenvalues of the Laplacian in $\X$ and by  $M(k)$ the maximal  number of nodal domains of any Laplacian eigenfunction of eigenvalue $\lambda_k.$ 
Then
\begin{equation}\label{eq: mainRCD}
\limsup_{k\to +\infty} \frac{M(k)}{k} \le \frac{(2\pi)^N}{\omega^2_N j^N_{\frac{(N-2)}{N}}}<1\, .
\end{equation}
\end{cor}

The result in Theorem \ref{main th} in the case of  Neumann eigenfunctions is interesting already when taking $(\X,\sfd,\hau^N)$ to be the $N$-dimensional Euclidean space. We extract this version below in a self-contained statement, for the convenience of the reader. 
\begin{cor}\label{cor: mainEucl}
Let $\Omega\subset \rr^N$, $N\ge2,$ be a uniform domain. For every $k\in \nn$ denote by $\lambda_k^\cN(\Omega)$ the Neumann Laplacian eigenvalues in $\Omega$ and by $M_\Omega^\cN(k)$ the maximal  number of nodal domains of a Neumann eigenfunction of eigenvalue $\lambda_k^\cN(\Omega).$   Then
\begin{equation}\label{eq: mainEucl}
\limsup_{k\to +\infty} \frac{M_\Omega^\cN(k)}{k} \le \frac{(2\pi)^N}{\omega^2_N j^N_{\frac{(N-2)}{N}}}<1.
\end{equation}
\end{cor}
Recall that the class of uniform domains in the Euclidean space includes bounded \emph{Lipschitz domains}, but also more irregular domains such as \emph{quasi disks} and in particular the interior of a \emph{Koch Snowflake} (see Section \ref{def: unif dom} for more details and references). 
A Pleijel theorem for Neumann eigenfunctions of Euclidean domains was firstly proved by Polterovich \cite{Pol09}, who considered planar domains with piecewise real analytic boundary. The general $N$-dimensional case was obtained in \cite{Le19} for domains $\Omega$ with $C^{1,1}$ boundary, where the regularity assumption is required in order to apply to eigenfunctions a reflection procedure across the boundary of $\Omega$. The same limitation on the regularity of the boundary appears in \cite{HaSh23} (in the context of more general Robin problems), where it is explicitly stated the problem of the validity of Pleijel theorem under a weaker regularity of the boundary. Very recently, the techniques introduced in \cite{Le19} were employed and refined in \cite{BCM23} in the planar case, where the authors were able to treat $2$-dimensional domains with smooth boundaries except for a finite number of vertices. 

Our work introduces a different strategy and avoids any reflection argument, allowing us to handle more general domains without imposing any restriction on the dimension.  To explain the basic idea of our method we recall that a key step in the original proof of the Pleijel theorem is to exploit the fact that an eigenfunction $u$ in $\Omega$,  when restricted to one of its nodal domains $U\subset \Omega$,  satisfies a zero-Dirichlet boundary condition in $U$ itself, thus allowing to apply the Faber-Krahn inequality and get a lower bound for the volume of $U$. While this is true for a Dirichlet-eigenfunction and for all its nodal domains, it is in general \textit{false} for a Neumann eigenfunction and a nodal domain that touches the boundary. The reflection procedure in \cite{Le19} is needed precisely  to handle this issue, but requires smoothness of the boundary. 
Instead our observation is that, by the very definition of nodal domain, an eigenfunction $u$ (even in the Neumann case) has indeed zero-Dirichlet boundary conditions in $U$ but \textit{relative to the ambient domain $\Omega$}, i.e.\ ignoring the portion of $\partial U$ which is contained in $\partial \Omega$.  The key point is then to view $\Omega$ as a metric space in its own right and prove that it is regular enough to satisfy a version of the Faber-Krahn inequality, which then allows to carry out the rest of the argument. This is where the uniform condition will enter into play ensuring the required analytical properties of $\Omega$.

Even if we use mostly techniques coming from the metric setting, we also develop some purely-Euclidean technical tools that we believe could be useful to show other Pleijel-type results in $\rr^N$ under low boundary-regularity. In particular, we prove 
a Faber-Krahn-tpye inequality  and a Green's formula for eigenfunctions of uniform domains (see Corollary \ref{cor:FKeucl} and Corollary \ref{cor:rayleig-on-nodal} respectively). Both the results were previously available only assuming $C^{1,1}$-boundary.

\vspace{3mm}
We now comment further on the assumptions and the proof of Theorem \ref{main th}.

The uniformity hypothesis on the domain   guarantees the discreteness of the spectrum of the Neumann Laplacian, a fact even needed to state the theorem. Additionally we will make a crucial use of analytical properties of uniform domains in metric measure spaces, such as Sobolev extension properties, the validity of a Poincar\'e inequality and a Sobolev embedding  (see Section \ref{sec:unif domain} for more details).

The non-collapsed assumption is more technical in nature, and we leave for future investigations the general case of possibly collapsed $\rcd(K,N)$ spaces. Let us notice that in collapsed $\RCD$ spaces the spectrum of the Laplacian can produce a singular and in some sense unexpected behaviour in the asymptotic regime (see the recent \cite{DHPW}), and thus this generalization seems non-trivial as we are going to further clarify in the next lines commenting the proof.

\vspace{3mm}
The main scheme of the proof of Theorem \ref{main th} is similar to the one usually employed in the smooth setting, e.g.\ in \cite{Le19,Stein14,Ple56,HaSh23,Donn14,BM82}. The two primary ingredients are the Weyl law and an almost-Euclidean Faber-Krahn inequality for small volumes. 

The Weyl law has already been investigated in the setting of $\RCD(K,N)$ spaces (see \cite{AHT} and \cite{ZZ}). In the non-collapsed case it takes the usual formulation
$$\lim_{\lambda\to +\infty} \frac{N(\lambda)}{\lambda^{N/2}}=\frac{\omega_N}{(2\pi)^N} \hau^N(\Omega)\,,$$
where $N(\lambda):=\sharp\{k\in \N : \lambda^{\cD}_k(\Omega)\le \lambda\}$ is the eigenvalues counting function and $\{\lambda^{\cD}_k(\Omega)\}_{k \in \nn}$ are the Dirichlet eigenvalues of the domain $\Omega$ (see Definition \ref{def:dir lapl}). We stress that the Weyl law in the  Dirichlet case is sufficient for our purposes, even if in our main statement we consider both Dirichlet and Neumann eigenfunctions. This thanks to the elementary inequality $\lambda_k^\cN(\Omega)\le \lambda_k^\cD(\Omega)$ between Neumann and Dirichlet eigenvalues (see Lemma \ref{lem:discrete spectrum}).
We remark that suitable forms of the Weyl law on the whole space have been studied under slightly more general assumptions than non-collapsing, but the situation is more intricate and there exist compact $\RCD(K,N)$ spaces for which $N(\lambda)$ is not asymptotic to $\lambda^{\beta}$ for any $\beta\ge 0$. We refer to \cite{AHT,DHPW} for  the details. 

Concerning the almost-Euclidean Faber-Krahn inequality, it roughly states that the first Dirichlet eigenvalue of an open set $U\subset \X$, of sufficiently small volume, is bounded below by the first Dirichlet eigenvalue of the Euclidean ball having the same volume and up to a small error. This will be obtained starting from an almost-Euclidean isoperimetric inequality for small volumes (similar to the one obtained in \cite{BM82} in the smooth setting)  and  rearrangement methods.  In contrast with the proof in the smooth case, our situation requires to deal with a  set $C$ of possibly “bad” points, and to work with sets $U$ that stay sufficiently far from $C$. We refer to Theorem \ref{thm:faber} and Theorem \ref{th:almost-iso} for the precise statements, and we suggest to compare them with \cite[Lemme 16,15]{BM82}. For both these results the non-collapsed assumption also plays a key role to ensure a more regular infinitesimal behaviour of the ambient space.

\section*{Acknowledgements}
The authors thank Asma Hassannezhad for helpful discussions.  The third author was supported by the Academy of Finland project Incidences on Fractals, Grant No.321896.

\section{Preliminaries}\label{sec:pre}

\subsection{Calculus  in metric measure spaces}\label{sec:calculus}
 The triple $\Xdm$ will denote a metric measure space, where $(\X,\sfd)$ is a complete and separable metric space and $\mm$ is a non-negative Borel measure, finite on bounded sets. We will also always assume  $\supp(\mm)=X$. For every set $A\subset \X$ we will denote by $\overline A$  its topological closure, by $A^c\coloneqq \X\setminus A$ its complement and by $\partial A$ its topological boundary. We denote by $B_r(x):=\{y\in \X: \sfd(x,y)<r\}$ the ball of radius $r$ and center $x$. The same set is also denoted by $B^{\X}_r(x)$ whenever we want to emphasize the role of the space $\X$. By $\sfd(A,B):=\inf\{\sfd(x,y) \ : \ x\in A, y\in B\}$ we denote the distance between two sets $A,B\subset \X$, so that $\sfd(A,\emptyset)=+\infty$. The open $\eps$-enlargement of a set $A\subset \X$ is denoted by $A^{\epsilon}:=\{x\in \X : \sfd(A,x)<\eps\}$. Given a set $C\subset \X$, we denote by $\sfd\restr{C}:=\sfd\restr{C\times C}$ the restriction of the distance to the set $C$. We will say that $(\X,\sfd)$ is proper if closed and bounded subsets of $\X$ are compact.
 
 Given a metric space $(\X,\sfd)$ and a rectifiable curve $\gamma:[a,b]\to \X$, we denote by $l(\gamma)$ its length (see e.g.\ \cite[Chapter 5.1]{HKST}).  We say that $\gamma$ joins $x\in \X$ and $y\in \X$ if $\gamma(a)=x$ and $\gamma(b)=y.$

\begin{definition}[Nodal domain]\label{def:nodal domains}
    Let $(\X,\sfd)$ be a metric space, $A\subset \X$ be any subset and $f:A \to \rr$ be a continuous function. The \emph{nodal domains} of $f$ (in $A$) are the connected components of $A \setminus \{x \in A \ : \ f(x)=0\}.$ 
\end{definition}

In the next result we recall some elementary properties of nodal sets.
\begin{lemma}\label{lem:basic nodal}
    Let $(\X,\sfd)$ be a metric space, $A\subset \X$ be any subset and $f:A \to \rr$ be a continuous function. Let $U\subset A$ be a nodal domain of $f.$ Then either $f>0$ or $f<0$ in $U.$ Moreover if $A$ is open and $(\X,\sfd)$ is locally connected then $U$ is also open.
\end{lemma}
\begin{proof}
The set $U$ is connected by definition, hence $f(U)\subset \rr$ is also connected and  does not contain zero. It follows that $f(U)\subset (0,\infty)$ or $f(U)\subset (-\infty,0).$ For a proof that $(\X,\sfd)$  locally connected implies that $U$ is open whenever $A$ is open see e.g.\ \cite[Theorem 25.3]{munkres}. 
\end{proof}

For every open set $\Omega\subset \X$ we denote by  $\LIP(\Omega)$, $\LIP_\loc(\Omega)$ and $\LIP_c(\Omega)$ respectively the space of Lipschitz functions, locally Lipschitz functions and Lipschitz functions with compact support in $\Omega$. We also denote by $\LIP_{bs}(\Omega)$ the subset of $\LIP(\X)$ of functions having support bounded and contained in $\Omega.$  
The slope $\lip(f)(x)$ of a locally Lipschitz function $f\in \LIP_\loc(\Omega)$ at a point $x\in \Omega$ is defined as
$$\lip(f)(x):=\limsup_{y\to x} \frac{|f(y)-f(x)|}{\sfd(y,x)},$$
taken to be 0 when $x$ is isolated.
The slope satisfies the following Leibniz rule: $\lip(fg)\le f\lip (g)+g\lip (f)$, for every $f,g \in \LIP_\loc(\Omega).$

Given $p\in [1,\infty]$, we use the notation $L^p(\X,\mm)$ (resp.\  $L^p_{\loc}(\X,\mm)$) for the space of Lebesgue $p$-integrable (resp.\  $p$-locally integrable) real functions on $X$ endowed with the Borel $\sigma$-algebra. For brevity, the same function space is also denoted by $L^p(\mm)$. When $\Omega\subset \X$ is an open set, we set $L^p(\Omega):=L^p(\Omega, \mm\restr \Omega)$ where $\mm\restr \Omega$ is the restriction of the measure $\mm$ to $\Omega$. For a function $u \in L^p(\Omega)$ we define its essential support $\supp(u)$ as the smallest closed set $C$ such that $u=0$ $\mea$-a.e.\ in $\Omega\setminus C.$

The Cheeger energy $\rmCh \colon L^2(\mm)\to [0,\infty]$ is defined as the convex and lower semicontinuous functional 
\[ \rmCh(f) := \inf \Big\{ \liminf_{n\to\infty} \int_{\X} \lip^2(  f_n)\, \d \mm \colon (f_n) \subset L^2(\mm)\cap \LIP_\loc(\X), \lim_{n\to\infty}\|f-f_n\|_{L^2(\mm)} =0 \Big\}.\]
The Sobolev space $\W\Xdm$ (or $\W(\X)$ for short) is then defined as  $\W\Xdm:= \{ \rmCh<\infty\}$  equipped with the norm $\| f\|^2_{\W(\X)} := \|f\|_{L^2(\mm)}^2 + \rmCh(f)$, which makes it  a Banach space. This approach to the definition of Sobolev space was introduced in \cite{AGS13}, where it is also shown to be equivalent to the previous definitions given in \cite{Cheeger00,Shanmugalingam00}. For every $f\in \W(\X)$ there exists a notion of modulus of the gradient called \emph{minimal weak upper gradient}, minimal w.u.g.\ for short, denoted by $|Df|\in L^2(\mm)$ and satisfying
\[
\rmCh(f)=\int_{\X} |Df|^2\d \mm\,.
\]
For every $f\in\LIP_\loc(\X)$ we have $|Df|\le \lip(f)$ $\mm$-a.e.. Moreover, the following calculus rules are satisfied (see e.g.\ \cite{GP20}): for every $f,g\in \W(\X)$ it holds
\begin{equation}\label{eq:loc and leib}
\begin{split}
	&\text{\emph{locality:}  $|D f|=|Dg|$ $\mea$-a.e.\ in $\{f=g\}$,}\\
	& \text{\emph{chain rule:} for every $\phi\in\LIP(\rr)$ with $\phi(0)=0$,   $\phi(f)\in \W(\X)$   and $|D\phi(f)|=|\phi'(f)||Df|,$}\\
 & \text{\emph{Leibniz rule:} for every $\eta \in \LIP\cap L^\infty(\X)$,    $\eta f\in \W(\X)$   and $|D(\eta f)|\le |\eta||Df|+|D\eta||f|.$}
\end{split}
\end{equation}
Given $\Omega\subset \X$ open we also define the following local Sobolev spaces
\begin{equation*}
\begin{split}
    &\W_0(\Omega)\coloneqq \overline{\LIP_{bs}(\Omega)}^{\W(\X)},\\
    &\W(\Omega)\coloneqq \{ f \in L^2(\Omega) \ : \ f\eta \in \W(\X), \, \forall \eta \in \LIP_{bs}(\Omega), \, |Df|\in L^2(\Omega)\},
\end{split}
\end{equation*}
where in the definition of $\W(\Omega)$ the minimal w.u.g.\  $|Df|\in L^2(\Omega)$ is defined by
\begin{equation}\label{eq:wug local}
  |Df|\coloneqq |D(f\eta_n)|, \quad \text{$\mea$-a.e.\ in $\{\eta_n=1\}$} ,
\end{equation} 
with $\eta_n\in \LIP_{bs}(\Omega)$ is any sequence satisfying $\{\eta_n=1\}\uparrow \Omega$ (there is no dependence on the chosen sequence, by the locality property of the minimal weak upper gradient). We endow $\W(\Omega)$ with the norm given by
\[
\|f\|_{\W(\Omega)}^2\coloneqq \|f\|_{L^2(\Omega)}^2+\||D f|\|_{L^2(\Omega)}^2, 
\]
which  makes it a Banach space. Observe that by the Leibniz rule we have that for every $f\in \W(\X)$ it holds that $f\restr\Omega \in \W(\Omega)$ and also $|D f|\restr\Omega=|D f\restr\Omega|$ $\mea$-a.e.\ in $\Omega$ (by the locality). Moreover,  for every $f\in\W_0(\Omega)$ we have $f=0$ $\mm$-a.e.\ in $X\setminus \Omega$ and thus $\|f\restr\Omega\|_{\W(\Omega)}=\|f\|_{\W(\X)}$, which shows that the map 
\begin{equation}\label{eq:def isoSob}
T:\W_0(\Omega)\rightarrow \W(\Omega), \qquad T(f):=f\restr\Omega
\end{equation} 
is a linear isometry. For these reasons, with a little abuse of notation, sometimes we identify $\W_0(\Omega)$ with $T(\W_0(\Omega))\subset \W(\Omega)$ and think to $f\in \W_0(\Omega)$ as an element of $L^2(\Omega)$.

If we choose $\Omega=\X$, then $\W(\X)=\W(\X,\sfd,\mea)$ with the same norm and minimal w.u.g., so the notation is consistent with the one given above.
\begin{remark}\label{eq:equivalence}
It can be shown (see e.g.\ \cite[Remark 2.15]{AH18}) that $\W(\Omega)$ coincides, up to $\mea$-a.e.\ equivalence of functions, with the Newtonian Sobolev space $N^{1,2}(\Omega,\sfd,\mea\restr \Omega)$ defined in \cite{Shanmugalingam00,Cheeger00} (see also \cite{BB13}). The norms of the two spaces coincide as well thanks to the equivalence proved in \cite{AGS13} between the various notions of minimal weak upper gradients. \fr
\end{remark}

Following \cite{Gigli12} we say that $\Xdm$ is \emph{infinitesimally Hilbertian} if $\W(\X)$ is a Hilbert space or equivalently if  the Cheeger energy  satisfies the parallelogram identity:
 \begin{equation}
     {\rm Ch}(f+g)+{\rm Ch}(f-g)=2{\rm Ch}(f)+2{\rm Ch}(g), \quad \forall f,g \in \W(\X).
 \end{equation}  
 If $\Xdm$ is infinitesimally Hilbertian, then $\W(\Omega)$ is a Hilbert space as well for every $\Omega\subset \X$ open (see e.g.\ \cite[Remark A.3]{CaRo22}). 
Moreover, we can give a notion of \emph{scalar product between gradients} of functions $f,g \in \W(\Omega)$ by setting
\begin{equation}\label{eq:def scalar}
    L^1(\Omega)\ni \nabla f\cdot \nabla g\coloneqq \frac12 \left(|D(f+g)|^2-|Df|^2-|Dg|^2\right),
\end{equation}
which is bilinear and satisfies
\begin{equation}\label{eq:C-S}
\begin{split}
        &|\nabla f\cdot \nabla g|\le |Df||Dg|,\quad \mea\text{-a.e.,}\quad \forall\, f,g \in \W(\Omega),\\
            &|\nabla f\cdot \nabla f|= |Df|^2,\quad \mea\text{-a.e.,}\quad \forall\, f \in \W(\Omega).
\end{split}
\end{equation}
 
Under the infinitesimally Hilbertian assumption we can define a notion of Laplacian via integration by parts.
\begin{definition}[{Neumann Laplacian}] \label{def:neum lapl} 
    Let $\Xdm$ be an infinitesimally Hilbertian metric measure space and $\Omega\subset \X$ be open. We say that $f\in \W(\Omega)$ belongs to the domain of the Neumann Laplacian, and we write $f\in \D(\Delta_\cN,\Omega)$,  if there exists  $h\in L^2(\Omega)$ such that 
\begin{equation}\label{eq:def Nlapl}
\int_{\Omega} hg \,\d\mm=-\int_{\Omega}\nabla f \cdot \nabla g\,\d\mm\,, \quad \forall\, g\in \W(\Omega).
\end{equation}
If $f\in\dom(\Delta_\cN,\Omega)$ then the function $h$ is unique and is denoted by $\Delta_\cN f.$
\end{definition}

\begin{definition}[{Dirichlet Laplacian}]\label{def:dir lapl}
    Let $\Xdm$ be an infinitesimally Hilbertian metric measure space and $\Omega\subset \X$ be open. Then $f\in \W_0(\Omega)$ belongs to the domain of the Dirichlet Laplacian, and we write $f\in \D(\Delta_\cD,\Omega)$,  if there exists  $h\in L^2(\Omega)$ such that 
\begin{equation}\label{eq:def Dlapl}
\int_{\Omega} hg \,\d\mm=-\int_{\Omega}\nabla f \cdot \nabla g\,\d\mm\,, \quad \forall\, g\in \W_0(\Omega).
\end{equation}
If $f\in\dom(\Delta_\cD,\Omega)$ then the function $h$ is unique and is denoted by $\Delta_\cD f.$
\end{definition}

Since under the infinitesimally Hilbertian assumption Lipschitz and bounded functions  are dense in $\W(\X)$ (see \cite{AGS13}) we have $\W_0(\X)=\W(\X)$ and so the Dirichlet and Neumann Laplacian coincide for $\Omega=\X $. In this situation we simply write $\Delta=\Delta_\cN=\Delta_\cD$ and call it simply Laplacian operator and write $f \in \dom(\Delta)$ in place of $f \in \dom(\Delta_\cN,\X)$ or $f \in \dom(\Delta_\cD,\X)$.

\begin{definition}[Eigenfunctions]\label{def:eigen}
    Let $\Xdm$ be an infinitesimally Hilbertian metric measure space and $\Omega\subset \X$ be open.
We say that a non-null $f\in \D(\Delta_\cD,\Omega)$ (resp.\  $\D(\Delta_\cN,\Omega)$) is a Dirichlet (resp.\  Neumann) \emph{eigenfunction of the Laplacian} in $\Omega$ of \emph{
eigenvalue} $\lambda \in\R$ if $\Delta_\cD f=-\lambda f$ (resp.\  $\Delta_\cN f=-\lambda f$). In the case $\Omega=\X$ we simply write that $f$ is an eigenfunction of the Laplacian of eigenvalue $\lambda.$
\end{definition}

\begin{remark}[Compatibility with Euclidean Laplacian]\label{rmk:eucl lapl}
    If $\Xdm=(\rr^N,|\cdot|,\Leb^N)$ and $\Omega\subset \rr^N$ is open, the spaces $\W(\Omega)$ and $\W_0(\Omega)$  coincide with the usual ones, also with the same norms, as shown in \cite[Theorem 4.5]{Shanmugalingam00} (see also \cite[Theorem A.2 and Corollary A.4]{BB13} or \cite[Section 2.1.5]{GP20}). In particular, by polarization,  the right-hand side of both \eqref{eq:def Nlapl} and \eqref{eq:def Dlapl} coincides with the integral of the scalar product between weak gradients in the classical sense. This shows that the definition of  eigenfunction (and eigenvalue) of the Neumann or Dirichlet Laplacian  in $\Omega$ given above coincides with the usual one in the Euclidean case.\fr
\end{remark}

For later use we observe that, whenever Lipschitz functions are dense in $\W(\X)$, for every bounded and open set $\Omega \subset \X$ it holds that
\begin{equation}\label{eq:support inside}
   \{ f \in \W(\X) \ : \  \sfd(\supp(f),\X\setminus\Omega)>0\} \subset \W_0(\Omega).
\end{equation}
Indeed there exist $\eta \in \LIP(\X)$  such that $\eta=1$ in $\supp (f)$ and $\supp(\eta)\subset \Omega$  and a sequence $f_n \in \LIP(\X)$ with $f_n \to f$ in $\W(\X)$,  by density. Then $\eta f_n\in \LIP_{bs}(\Omega)$ and $\eta f_n \to f$ in $\W(\X)$,  which shows that $f \in \W_0(\Omega).$

We state in the next lemma an inequality between Neumann and Dirichlet eigenvalues that will play a key role in the sequel. Note that in the statement by $\W_0(\Omega)\hookrightarrow L^2(\Omega)$ we mean, more precisely, that $T(\W_0(\Omega))\hookrightarrow L^2(\Omega)$ where $T$ is defined in \eqref{eq:def isoSob}. 
\begin{lemma}\label{lem:discrete spectrum}
Let $\Xdm$ be an infinitesimally Hilbertian metric measure space and let $\Omega\subset \X$ be open. Let us suppose that $\W_0(\Omega)\hookrightarrow L^2(\Omega)$ with compact inclusion. Then $-\Delta_\cD$ has discrete spectrum, i.e. the eigenvalues form a diverging sequence (counted with multiplicity) that we denote by
\begin{equation}\label{eq: Dirlisteigen}
0\leq\lambda_1^\cD(\Omega)\le \lambda_2^\cD(\Omega)\le \dots \lambda_k^\cD(\Omega)\le \dots \rightarrow +\infty\,.
\end{equation}
If moreover $\W(\Omega)\hookrightarrow L^2(\Omega)$ with compact inclusion, then also
$-\Delta_\cN$ has discrete spectrum denoted by
\begin{equation}\label{eq: Neumlisteigen}
0=\lambda_1^\cN(\Omega)\le \lambda_2^\cN(\Omega)\le \dots \lambda_k^\cN(\Omega)\le \dots \rightarrow +\infty\,,
\end{equation}
 and it holds
\begin{equation}\label{eq: Neum<Dir}
\lambda_k^\cN(\Omega)\le \lambda_k^\cD(\Omega)\,, \quad \forall k\in\N.
\end{equation}
\end{lemma}

\begin{proof}
Let us introduce the local Cheeger energies 
\begin{equation*}\label{eq:DirForms}
\begin{aligned}
        &\rmCh_{\cD}^{\Omega}:L^2(\Omega)\rightarrow [0,+\infty], \ \qquad \rmCh_{\cD}^{\Omega}(f):=\begin{cases}\int_{\Omega}|Df|^2\d\mm \qquad &\textrm{if} \ f=g\restr{\Omega} \, \textrm{for some} \ g\in W_0^{1,2}(\Omega)\,,\\
        +\infty &\textrm{otherwise},
        \end{cases}
        \\
        &\rmCh_{\cN}^{\Omega}:L^2(\Omega)\rightarrow [0,+\infty], \qquad \rmCh_{\cN}^{\Omega}(f):=\begin{cases}\int_{\Omega}|Df|^2\d\mm \qquad &\textrm{if} \ f\in \W(\Omega)\,,\\
        +\infty &\textrm{otherwise},
        \end{cases}
\end{aligned}
\end{equation*}
and notice that they define two Dirichlet forms, i.e. two densely defined, Markovian, closed, quadratic forms \cite{BH91,FOT11}. To check this, it is sufficient to recall the calculus rules given in \eqref{eq:loc and leib} and, for the $L^2$-lower semicontinuity, the equivalent definition through relaxation (see \cite{CaRo22} for all the details). We denote by $L_{\cD}$ (resp.\  $L_{\cN}$) the infinitesimal generator of $\rmCh_{\cD}^{\Omega}$ (resp.\  $\rmCh_{\cN}^{\Omega}$) with its associated domain $\D(L_{\cD})$ (resp.\  $\D(L_{\cN})$). Notice that, by the very definition, $\D(\Delta_\cN,\Omega)=\D(L_{\cN})$ with $\Delta_\cN=L_{\cN}$. Regarding the Dirichlet Laplacian, we have $f\in \D(\Delta_\cD,\Omega)$ if and only if $f\restr{\Omega}\in\D(L_{\cD})$ with $\Delta_\cD f=L_{\cD}(f\restr{\Omega})$. In particular, $\lambda$ is an eigenvalue of $-\Delta_\cD$ (resp.\  $-\Delta_\cN$) if and only if it is and eigenvalue of $-L_{\cD}$ (resp.\  $-L_{\cN}$). 

From the classical theory of Dirichlet forms \cite{BH91,FOT11} we know that $-L_\cD$ and $-L_\cN$ are non-negative, densely defined, linear, self-adjoint operators on $L^2(\Omega)$. Under these assumptions, it is well known (see e.g. \cite{Dav95}) that the compactness of the embedding of $\W_0(\Omega)$ (resp.\  $\W(\Omega)$) in $L^2(\Omega)$ implies the discreteness of the spectrum of $-L_{\cD}$ (resp.\  $-L_{\cN}$) and thus of $-\Delta_\cD$ (resp.\  $-\Delta_\cN$).

Since $T(\W_0(\Omega))\subset \W(\Omega)$ as Hilbert spaces, we also know that whenever $\W(\Omega)\hookrightarrow L^2(\Omega)$ with compact inclusion both the spectra are discrete. 

We also have at disposal the variational characterization of the eigenvalues , see e.g. \cite[Theorems 4.5.1, 4.5.3]{Dav95}. More precisely, defined 
\begin{equation}\label{eq: Rayleigh}
\begin{aligned}
&\lambda^{\cN}(\Omega)[M]:=\sup\{\rmCh_{\cN}^{\Omega}(f)\ : \ f\in M, \|f\|_{L^2(\Omega)}=1\}, \\ 
&\lambda^{\cD}(\Omega)[M]:=\sup\{\rmCh_{\cD}^{\Omega}(f)\ : \ f\in M, \|f\|_{L^2(\Omega)}=1\},
\end{aligned}
\end{equation}
we know that for every $k\in \nn$
\begin{equation}\label{eq: inf-sup-eigen}
\begin{aligned}
&\lambda_k^\cN(\Omega)=\inf\{\lambda^{\cN}(\Omega)[M]\ : \ M\subset \W(\Omega), \ \dim(M)=k\}, \\
&\lambda_k^\cD(\Omega)=\inf\{\lambda^{\cD}(\Omega)[M]\ : \ M\subset T(\W_0(\Omega)), \ \dim(M)=k\},
\end{aligned}
\end{equation}

The inequality \eqref{eq: Neum<Dir} thus follows immediately from \eqref{eq: inf-sup-eigen} since the infimum is taken on a larger set and $\rmCh_{\cN}^{\Omega}(f)= \rmCh_{\cD}^{\Omega}(f)$ for every $f\in T(\W_0(\Omega)).$
\end{proof} 

We will use in the sequel the notation introduced in the previous lemma, i.e.\ whenever $-\Delta_\cD$ (resp.\  $-\Delta_\cN$) has discrete spectrum in $\Omega$ we will denote by $\{\lambda_k^\cD(\Omega)\}_{k\in \nn}$ (resp.\  $\{\lambda_k^\cN(\Omega)\}_{k\in \nn}$) the sequence of its eigenvalues. In the case $\Omega=\X$, assuming the discreteness of the spectrum of $-\Delta$, we will simply write $\lambda_k$ in place of $\lambda_k^{\cN}(\X)$.

For an arbitrary m.m.s.\ $\Xdm$ and any $\Omega \subset \X$ open subset we also introduce 
\begin{equation}\label{def: lambda1}
\lambda_1(\Omega)\coloneqq \inf\left\{\frac{\int |D u|^2\d \mm}{\int u^2\d \mm} : u \in \LIP_{bs}(\Omega), \ u\not\equiv 0 \right\}
\end{equation}
and we call $\lambda_1(\Omega)$ the \emph{first eigenvalue of the Laplacian on $\Omega$ with zero Dirichlet boundary conditions}.  
Recalling the definition  definition of $\W_0(\Omega)$ we have the following characterization of $\lambda_1(\Omega)$:
\begin{equation}\label{char: lambda1}
\lambda_1(\Omega)=\inf\left\{\frac{\int |D u|^2\d \mm}{\int u^2\d \mm} : u \in \W_0(\Omega), \ u\not\equiv 0 \right\}.
\end{equation}
Note that differently from $\lambda_1^\cD(\Omega)$, which we defined only when the inclusion $\W_0(\Omega)\hookrightarrow L^2(\Omega)$ is compact, $\lambda_1(\Omega)$ is always defined. Nevertheless, even if not needed, we stress that  whenever $\lambda^{\cD}_1(\Omega)$ exists we do have
\[
    \lambda_1(\Omega)=\lambda^{\cD}_1(\Omega),
\]
as follows by \eqref{char: lambda1} and \eqref{eq: inf-sup-eigen}.

\subsection{Sets of finite perimeter}

Let $f\in L^1_{\loc}(\X,\mm)$ and let $U\subset \X$ be open. Following \cite{MR03,AD14} we define 
\begin{equation}\label{def: per}
|{\bf D}f|(U):=\inf\left\{\liminf_{n\to \infty} \int_U \lip(f_n)\,\d\mm \ : \ f_n\in \LIP_{\loc}(U), \ f_n\to f \ \textrm{in} \ L^1_{\loc}(U,\mm)\right\}\,,
\end{equation}
and we say that $f$ is of locally bounded variation if $|{\bf D}f|(U)<+\infty$ for every $U$ open and bounded.  We also set 
$$ |{\bf D}f|(A)\coloneqq \inf \{|{\bf D}f|(U) \ : \ U\subset \X \text{ open, } A\subset U\}, \quad \forall\, \text{$A\subset \X$ Borel} $$
(note that this coincides with \eqref{def: per} if $A$ is open).
For every Borel set $E\subset \X$ and $A\subset \X$ Borel we define $\Per(E,A):=|{\bf D}\chi_E|(A)<+\infty$, where $\chi_E:X\to \{0,1\}$ denotes the characteristic function of $E$. We say that $E$  is of  finite perimeter if $\Per(E)\coloneqq \Per(E,\X)<+\infty$. 

When $f$ is of locally bounded variation (respectively, $E$ is a set of  finite perimeter), the map $A\mapsto|{\bf D}f|(A)$ (respectively, $A\mapsto\Per(E,A)$) defines a Borel measure (see \cite{AD14,MR03}).   Every $f\in \LIP(\X)$ is of locally bounded variation and $|{\bf D}f|\le \lip (f) \mm$ (see \cite[Remark 5.1]{AD14}).

From the definitions it immediately follows  that $E$ is of finite perimeter if and only if $E^c$ is of finite perimeter, in which case $\Per(E,\cdot)=\Per(E^c,\cdot)$ holds. 

In the sequel we will take advantage of the following coarea-type inequality.
\begin{proposition}\label{prop: coarea}
Let $\Xdm$ be a metric measure space and fix $x \in \X.$ Then for a.e.\ $r>0$ the ball $B_r(x)$ has finite perimeter and for every Borel set $A\subset \X$ it holds
\begin{equation}\label{eq:coarea}
    \int_0^R \Per(B_r(x),A) \, \d r\le \mea(B_R(x)\cap A), \quad \forall\, R>0.
\end{equation}
\end{proposition}
\begin{proof}
    Since the function $\sfd_x(\cdot)\coloneqq \sfd(x,\cdot)$ is 1-Lipschitz, it is of locally bounded variation and $|{\bf D}\sfd_x|\le \lip (\sfd_x )\mm\le \mm. $ Then by  the coarea formula (see Proposition 4.2  in \cite{MR03}) we get directly that $B_r(x)=\{\sfd_x(\cdot)<r\}$ has finite perimeter for a.e.\ $r>0$ and that 
    \[
    \int_0^{+\infty} \Per(B_r(x),E)\d r=|{{\bf D} \sfd_x}|(E)\le \mea(E), \quad \forall\,\, E\subset\X \text{ Borel}. 
    \]
    Then \eqref{eq:coarea} follows taking $E\coloneqq B_R(x)\cap A$ and observing that by the very definition in \eqref{def: per} it holds $\Per(B_r(x),B_R(x))=0$ for every $r>R$.
\end{proof}

Let $\Xdm$ be a metric measure space. Given a Borel set $E\subset \X$ we define the upper and lower densities at $x$ as
\[
\overline D(E,x)\coloneqq \limsup_{r\to 0^+}\frac{\mea(B_r(x)\cap E)}{\mea(B_r(x))}, \quad \underline D(E,x)\coloneqq \liminf_{r\to 0^+}\frac{\mea(B_r(x)\cap E)}{\mea(B_r(x))}.
\]
Clearly, if  $x\in \X$ is such that $\overline D(E,x)>0$, by definition of limit superior we have $x\in \overline{E}$ (since every open ball with center $x$ must intersect $E$).

The \emph{essential boundary} and the \emph{essential interior} are given respectively by
\begin{align*}
	&\essb E\coloneqq \{x \in \X \ : \ \overline D(E,x)>0,\, \overline D(E^c,x)>0 \},\\
	&E^{(1)}\coloneqq\{x \in \X \ : \ \overline D(E,x)=\underline D(E,x)=1\},
\end{align*}
which are both Borel sets. As a direct consequence of the definition of these sets, notice that if $E\subset F$, then $E^{(1)}\subset F^{(1)}$. Moreover, $\essb E=\essb (E^c)$.

We collect in the next lemma all the elementary facts that we will need about the essential boundary and the essential interior. 
\begin{lemma}\label{lem: essb and essi basic}
Let  $(\X,\sfd)$ be a metric space and $E,F\subset \X$ be Borel sets. We have the following:
\begin{enumerate}[label=(\roman*)]
\item If $E$ is open,  $E\subset E^{(1)}$.
\item $\essb E\subset \partial E$.
\item $(E\cap F)^{(1)}\subset E^{(1)}\cap F^{(1)}$.
\item $(E^{c})^{(1)}\subset (E^{(1)})^c$. In particular, if $E$ is  open then $(E^{c})^{(1)}\subset E^c$.
\item If $E$ and $F$ are disjoint, then also $E^{(1)}$ and $F^{(1)}$ are disjoint.
\item $\essb (E\cap F)\cup\essb (E \cup F)\subset \essb E \cup \essb F$.
\end{enumerate}
\end{lemma}
\begin{proof}
\begin{enumerate}[label=\textit{(\roman*)}]
\item Let $x\in E$. Since $E$ is open, $B_r(x)\subset E$ for sufficiently small $r>0$, thus $\underline D(E,x)=1$ and $x\in E^{(1)}$.
\item As we have already observed, if $x\in \essb E$ it holds $x\in \overline{E}$ and $x\in \overline{E^c}$, thus $x\in \partial E$. 
\item The result is a direct consequence of the fact that $\underline D(E\cap F,x)\le \underline D(E,x)\le 1$ for every $x\in \X$.
\item Let $x\in (E^{c})^{(1)}$, i.e. $\underline D(E^c,x)=1$. In particular, 
$$\frac{2}{3}\mm(B_r(x))<\mm(B_r(x)\cap E^c)=\mm(B_r(x))-\mm(B_r(x)\cap E)$$
for sufficiently small $r>0$. Thus $\mm(B_r(x)\cap E)<\frac{1}{3}\mm(B_r(x))$ for $r>0$ small enough, which implies $x\notin E^{(1)}$. The second conclusion follows from what we have just proven and point $({i})$. 
\item The assumption $E\cap F=\emptyset$ is equivalent to $E\subset F^{c}$. Passing to the essential interior it holds $E^{(1)}\subset (F^{c})^{(1)}$ and using $({iv})$ one deduces that $E^{(1)}\subset (F^{(1)})^c$ which gives the desired conclusion.
\item This is proven e.g.\ in \cite[Prop. 1.16]{BPRT} (note that the doubling assumption on $\mm$ is not used in that  statement).
\end{enumerate}
\end{proof}

We conclude this part with the following elementary and well known result. Since we could not find it stated exactly in this form in the literature, we include a proof.
\begin{lemma}\label{lem:perimeter inside}
    Let $\Xdm$ be a metric measure space and let $C\subset \X$ be closed. Then for every $E\subset C$ Borel satisfying $\sfd(E,\X\setminus C)>0$ it holds
\[
\Per(E)=\Per_{C}(E),
\]
where $\Per_{C}(E)$ denotes the perimeter of $E$ computed in the metric measure space $(C,\sfd\restr{C},\mea\restr C).$
\end{lemma}
\begin{proof}
    First observe that for every $f \in \LIP_\loc(\X)$ it holds $f\restr {C} \in \LIP_\loc(C)$ and $\lip_{C} (f) \le \lip (f)$, where $\lip_{C} (f)$ denotes the slope of $f$ computed in the metric space $(C,\sfd\restr{C})$. This and the definitions implies $\Per_{C}(E)\le \Per(E).$ For the other inequality it is sufficient to find a sequence $f_n \in \LIP_{\loc}(C,\sfd\restr{C})$ such that  $\sfd(\supp(f_n),\X \setminus C)>0$, $f_n \to \nchi_{E}$ in $L^1(\mm)$ and  $\int_C \lip_C(f_n)\d \mm\to \Per_{C}(E)$. Indeed extending $f_n$ by zero to the whole $\X$ we  have $\int_\X \lip(f_n)\d \mm=\int_C \lip_C(f_n)\d \mm$, so that 
    \begin{equation*}
    \Per(E)\leq \lim_{n\to +\infty}\int_{X}\lip(f_n)\d \mm=  \lim_{n\to +\infty}\int_C \lip_C(f_n)\d \mm= \Per_{C}(E).
    \end{equation*}
    To produce such sequence we consider a sequence $g_n \in \LIP_\loc(C,\sfd\restr{C})$ such that $g_n\to \nchi_E$ in $L^1(C,\mm\restr{C})$ and $\int_{C} \lip_{C}(g_n)\d \mm\to \Per_{C}(E)$, which exists by definition.  Then we take any $\eta \in \LIP_{\loc}(C)$ satisfying  $\sfd(\supp(\eta),\X \setminus C)>0$, $\eta=1$ in a neighbourhood of $E$, $0\le \eta \le 1$ and set $f_n\coloneqq \eta g_n \in \LIP_{\loc}(C,\sfd\restr{C}).$ Clearly $f_n\to \nchi_E$ in $L^1(C;\mm\restr{C})$.  Moreover by the Leibniz rule for the slope $\lip_{C}(f_n)\le g_n\lip_{C} (\eta )+\lip_{C}(g_n)$. In addition $\lip_{C}(\eta)=0$ in $E$. Therefore
    \begin{align*}
    &\Per_{C}(E)\le  \liminf_n \int_{C}\lip_{C}(f_n)\d \mm\le \limsup_n \int_{C}\lip_{C}(f_n)\d \mm\\
    &\le \limsup_n  \int_{C}\lip_{C}(g_n)\d \mm+ \lip_{C} (\eta) \int_{\X\setminus E} g_n \d \mm=\Per_{C}(E),
    \end{align*}
    where the second term vanishes because $g_n\to \nchi_E$ in $L^1(C;\mm\restr{C})$.
\end{proof}

\subsection{PI spaces }\label{sec:pi}
Most of the arguments along the note will be carried out in the general setting of locally doubling m.m.\ spaces supporting a Poincar\'e inequality, also called PI spaces. We refer to \cite{BB13,HKST} and references therein for a thorough introduction on this topic and recall here only the properties of these spaces that will be used in this note.
\begin{definition}[PI space]\label{def:PI}
A metric measure space $\Xdm$ is said to be a \emph{PI space} if:
\begin{enumerate}[label=\roman*)]
    \item it is \emph{uniformly locally doubling}, i.e.\ if there exists a function $C_D\, :(0,\infty)\to (0,\infty)$ such that
\begin{equation*}
\mm\big(B_{2r}(x)\big)\leq C_D(R)\,\mm\big(B_r(x)\big)\,,\quad\text{for every }0<r<R\text{ and }x\in\X,
\end{equation*}
\item supports a \emph{weak local \((1,1)\)-Poincar\'{e} inequality}, i.e   there exists a constant $\lambda\ge 1$ and a function $C_P\, :(0,\infty)\to (0,\infty)$ such that for any $f\in \LIP_{\loc}(X)$ it holds
\begin{equation*}
\fint_{B_r(x)}\bigg|f-\fint_{B_r(x)}f\,\d\mm\bigg|\,\d\mm\leq C_P({R})\,r\fint_{B_{\lambda r}(x)}\lip(f)\,\d\mm\,,
\quad\text{for every }0<r<R\text{ and }x\in\X.
\end{equation*}
\end{enumerate}
\end{definition}
Observe that the uniformly locally doubling assumption implies that  PI spaces are proper.
Additionally PI spaces are connected and locally connected  (see e.g.\ \cite[Theorem 4.32]{BB13} or \cite[Prop. 4.8]{BB18}).

We will need the following approximation result, which is a variation of \cite[Lemma 3.6]{MS20}.
\begin{lemma}[Approximation with non-vanishing slope]\label{lem:nonvanishing gradients}
    Let $\Xdm$ be a bounded PI space. Then for every open subset $\Omega\subset \X$ and any non-negative $u \in \LIP_{c}(\Omega)$ there exists a sequence of non-negative $u_n \in \LIP_{c}(\Omega)$ satisfying $\lip(u_n)\neq0 $ $\mm$-a.e.\  in $\{u_n>0\}$ and such that $u_n \to u$ in $\W(\X).$
\end{lemma}
\begin{proof}
Since $\Xdm$ is a bounded PI space it admits a geodesic distance $\tilde \sfd$ that is bi-Lipschitz equivalent to $\sfd,$ i.e.\ $L^{-1}\tilde \sfd \le \sfd\le L \tilde \sfd $ for some constant $L\ge 1$ (see e.g.\  \cite[Corollary 8.3.16]{HKST}).
Hence we can apply \cite[Lemma 3.6]{MS20}  to deduce that the conclusion of the lemma holds in the m.m.s.\ $(\X,\tilde \sfd,\mm)$. However $u \in \LIP_c(\Omega;\sfd)$ if and only if $\LIP_c(\Omega;\tilde \sfd)$, with  $\lip(u)\ge L^{-1}\tilde \lip(u)$,  where $\tilde \lip(\cdot)$ denotes the slope computed in the metric space $(\X,\tilde \sfd)$. Hence if $\tilde \lip(u)\neq 0$ $\mea$-a.e.\ in $\{u>0\}$, then $\lip(u)\neq 0$ in $\{u>0\}$ $\mea$-a.e.. Moreover, since $\LIP_\loc(\X,\sfd)=\LIP_\loc(\X,\tilde \sfd)$ and  $\lip(\cdot)\le L\tilde \lip(\cdot)$ we have that $\|u\|_{\W\Xdm}\le L\|u\|_{\W(\X,\tilde \sfd,\mm)}$ for all $u \in \LIP_{bs}(\X,\sfd)$. Therefore the conclusion holds also for the m.m.s.\ $\Xdm.$ 
\end{proof}

The above lemma allows to give the following characterization of $\lambda_1(\Omega)$.

\begin{lemma}[Characterization of $\lambda_1(\Omega)$ via functions with non-vanishing slope]
  	Let $\Xdm$ be a bounded PI space and $\Omega \subset \X$ be open. Then
   \begin{equation}\label{eq:non-van lambda1}
\lambda_1(\Omega)=\inf\left\{\frac{\int |D u|^2\d \mm}{\int u^2\d \mm} : u \in \LIP_{c}(\Omega), \ u\not\equiv 0, \ u\ge 0,\ \lip(u)\neq0 \text{ $\mm$-a.e.\  in $\{u>0\}$} \right\}\,  .
\end{equation}
  
\end{lemma}
\begin{proof}
For every for every $u \in \LIP_{bs}(\Omega)$, setting $\tilde u\coloneqq |u|$ we have $\tilde u \in \LIP_{bs}(\Omega)$ and by the chain rule (see \eqref{eq:loc and leib}) also that $|D\tilde u|=|Du|$ $\mea$-a.e.. This shows that \eqref{eq:non-van lambda1} holds if we remove the requirement that $\lip(u)\neq0$ $\mm$-a.e.\  in $\{u>0\}$. From this, to get the validity of the full \eqref{eq:non-van lambda1} it is sufficient to apply Lemma \ref{lem:nonvanishing gradients}.
\end{proof}

We recall the following deep result proved in \cite{Cheeger00}, relating the notions of minimal weak upper gradient and slope, in the setting of PI spaces . 
\begin{theorem}\label{thm:lip pi}
    Let $\Xdm$ be a PI space. Then 
    \begin{equation}\label{eq:lip pi}
        \lip(f)=|Df|,\quad \mea\text{-a.e., for every $f\in \LIP_{bs}(\X)$.}
    \end{equation}
\end{theorem}
The following is a consequence of the Rellich-Kondrachov compactness theorem in PI spaces.

\begin{theorem}[\hspace{1sp}{\cite[Theorem 8.3]{HK00}}]\label{thm:compact embedding}
    Let $\Xdm$ be a PI space and $\Omega\subset \X$ be open and bounded. Then the embedding $\W_0(\Omega)\hookrightarrow L^2(\Omega)$ is compact.
\end{theorem}
From Theorem \ref{thm:compact embedding} and the discussion in Section \ref{sec:calculus} we deduce that 
\begin{equation}\label{eq:PI discrete sepctrum}
    \parbox{12cm}{The Dirichlet Laplacian has discrete spectrum on any bounded open subset of an infinitesimally Hilbertian PI space.}
\end{equation}

The following result is well known. In particular the proof can be achieved by a standard Moser iteration scheme (see e.g.\ \cite[Theorem 8.24]{GilTru}), which is available in a PI space.  Indeed, as firstly observed in  \cite{Saloff} (see also \cite{HK00}), a Poincar\'e inequality and a doubling condition together imply a Sobolev inequality, which is then sufficient  to perform the Moser scheme  (see, for example, \cite[Chapter 8]{BB13} or \cite{BjMa}).

\begin{theorem}[Continuity of eigenfunctions]\label{thm:continuous eigen}
    Let $\Xdm$ be an infinitesimally Hilbertian PI space,  $\Omega\subset \X$ be open  and $u$ be a Dirichlet or Neumann eigenfunction of the  Laplacian in $\Omega$. Then $u$ is locally H\"older continuous in $\Omega.$ 
\end{theorem}

We now pass to the properties of sets of finite perimeter in the  setting of PI spaces .
As the measure $\mm$ is locally doubling, the Lebesgue's differentiation theorem holds (see e.g. \cite[Section 3.4]{HKST}) hence  we have $\mm(E\triangle E^{(1)})=0$ for every Borel set $E$, where $E\triangle E^{(1)}:=(E\setminus E^{(1)})\cup(E^{(1)}\setminus E)$ denotes the symmetric difference between $E$ and $E^{(1)}$. Moreover  by \cite[Theorem 5.3]{Amb02} we have that for every set of finite perimeter $E\subset \X$ the measure $P(E,\cdot)$ is concentrated on $\partial^e E.$ In particular we get
\begin{equation}\label{eq:ball perimeters}
    \Per(B^c,.)=\Per(B,.)=\Per(B,.)\restr{B^c},
\end{equation}
for every ball $B=B_r(x)\subset \X$ having finite perimeter, having used   $\partial^e B_r(x)\subset \partial B_r(x)\subset (B_r(x))^c$ (see $(ii)$ in Lemma \ref{lem: essb and essi basic}).

It is well known that every PI space admits an isoperimetric inequality (see \cite{Amb02,HK00,MR03}). We report in the following statement a simplified version sufficient to our purposes.
\begin{proposition}[Isoperimetric inequality for small volumes]\label{prop: class isop}
Let $\Xdm$ be a bounded PI space satisfying for some constant $s>1$ and $c>0$
\begin{equation}
    \frac{\mea(B_r(x))}{\mea(B_R(x))}\ge c \left(\frac{r}{R}\right)^s, \quad \forall\, x \in \X,\quad \forall\, 0<r<R.
\end{equation}
Then there exist constants $w_0=w_0(\X)>0$ and $C_I=C_I(\X,s)>0$ such that 
\begin{equation}
\Per(E) \ge C_I \mm(E)^{\frac{s-1}s}, \quad \forall\,   E \subset \X \text{ Borel  such that $\mea(E)\le w_0$}.
\end{equation}
\end{proposition}
\begin{proof}
By  \cite[Theorem 4.3]{Amb02} there exist constants $\sigma=\sigma(\X)\ge 1$ and $C=C(\X,s)>0$ such that
\[
\Per(E,B_{\sigma r}(x)) \ge C\frac{\mea(B_r(x))^\frac1s}{r}\min\left(\mea(B_r(x)\cap E),\mea(B_r(x)\setminus E)\right)^\frac{s-1}{s},\quad \forall x \in \X,\quad \forall r>0.
\]
Taking $r\coloneqq \diam(\X)$ and $w_0\coloneqq \mea(\X)/2$ the result follows (recall that $\mea(\X)<+\infty$ since $\X$ is bounded).
\end{proof}

We conclude this part reporting the following technical result.
\begin{prop}[\hspace{1sp}{\cite[Lemma 2.6]{APPV23}}]\label{prop:per intersection}
	Let \((\X,\sfd,\mm)\) be a PI space. Let \(E,F\subseteq\X\) be sets of  finite perimeter with \(P(E,\partial^e F)=0\). Then
\[
P(E\cap F,\cdot)\leq P(E,\cdot) \restr {F^{(1)}} +P(F,\cdot) \restr {E^{(1)}}.
\]
\end{prop}

\subsection{P\'olya-Szeg\H{o} inequality in metric measure spaces}\label{sec:ps}
Here we report a version of the P\'olya-Szeg\H{o} rearrangement inequality for metric measure spaces, following \cite{MS20,NV21}.
The main difference with the classical version in the Euclidean space \cite{PS51} is that, even if the initial function lives in a metric space, the symmetrization will be defined in $\rr^N$. The result is a generalization of the P\'olya-Szeg\H{o} inequality  introduced in \cite{BM82} in the case of Riemannian manifolds.

\begin{definition}[Distribution function]\label{def:distributionfunction}
	Let $\Xdm$ be a metric measure space, $\Omega\subseteq \X$  be an open set with $\mm(\Omega)<+\infty$ and $u:\Omega\to[0,+\infty)$ be a non-negative Borel function. We define  $\mu:[0,+\infty)\to[0,\mm(\Omega)]$, the distribution function of $u$, as 
	\begin{equation}\label{eq:defdistr}
		\mu(t):=\mm(\{u> t\}).
	\end{equation}
\end{definition}
For $u$ and $\mu$ as above, we let $u^{\#}$ be the generalized inverse of $\mu$, defined by
\begin{equation*}
	u^{\#}(s):=
	\begin{cases*}
		{\rm ess}\sup u & \text{if $s=0$},\\
		\inf\left\lbrace t:\mu(t)<s \right\rbrace &\text{if $s>0$}. 
	\end{cases*}
\end{equation*}
It can be checked that $u^\#$ is non-increasing and left-continuous.

 Next, we define the \emph{Euclidean monotone rearrangement} into the Euclidean space $(\R^N, |\cdot|, \Leb^N)$, where $\Leb^N$ is the $N$-dimensional Lebesgue measure. From now on, we denote by $\omega_N\coloneqq \Leb^N(B_1(0))$ the Lebesgue measure of the unit ball in the Euclidean space $\R^N$.
\begin{definition}[Euclidean monotone rearrangement] \label{def:eucl rearrangement}
    Let $\Xdm$ be a metric measure space and $\Omega \subset \X$ be open with $\mm(\Omega)<+\infty$ and $N \in \N.$
    For any Borel function $u\colon \Omega \rightarrow \R^+$, we define $\Omega^*:= B_r(0)\subset \R^N$, choosing $r>0$ so that $\Leb^N(B_r(0))= \mm(\Omega)$ (i.e. $r^N=\omega_N^{-1}\mm(\Omega)$) and  the monotone rearrangement $u^*_{N} : \Omega^* \rightarrow \R^+$ by
    \[ u^*_{N}(x) := u^\#(\Leb^N(B_{|x|}(0)))=u^\#(\omega_N|x|^N ), \qquad \forall x \in \Omega^*.\]
\end{definition}
In particular $u$ and $u_N^*$ are equimeasurable, i.e.\ $\mm(\{u>t\})=\Leb^N(\{u^*_N>t\})$ for all $t>0.$
In the sequel, whenever we fix $\Omega$ and $u \colon \Omega \to [0,\infty)$,  the set $\Omega^*$ and the rearrangement $u^*_{N}$ are automatically defined as  above. 
 Observe also that, given $u \in L^2(\Omega)$, its monotone rearrangement must be defined by fixing a Borel representative of $u$. However, this choice does not affect the outcome object $u^*_{N}$, as clearly the distribution function $\mu(t)$ of $u$ is independent of the representative.

The following result is essentially contained in \cite{NV21}, see in particular \cite[Remark 3.7]{NV21},  (see also \cite{MS20} for a similar result), since the only difference is that here the rearrangement is defined in $\R^N$ instead that on an interval. Nevertheless we include a short argument outlining the main points of the proof.
\begin{theorem}[Euclidean P\'olya-Szeg\H{o} inequality]\label{thm:eu_polyaszego}
		Let $(\X,\sfd,\mm)$ be a bounded PI space,  $\Omega\subsetneq \X$ be open and fix $N \in \N\setminus\{1\}.$ Suppose there exists a constant $\tilde{C}>0$ such that
		\begin{equation} \label{eq:isop_eu}
\Per(E) \ge \tilde{C}\mm(E)^{\frac{N-1}N}\, , \quad \forall\,   E \subset \Omega \text{ Borel.}
\end{equation}
  Then:
\begin{enumerate}[label=\roman*)]
    \item For every  $u \in \LIP_c(\Omega)$ non-negative, $u\not\equiv 0$, with $\lip(u)\neq 0$ $\mm$-a.e.\ in $\{u>0\}$, then $u^*_N \in \LIP_c(\Omega^*)$ and it holds
    \begin{equation}\label{eq:improved polya}
        \int_{\{u\le s\}} |Du|^2\, \d \mm \ge\int_0^s\left(\frac{\Per (\{u>t\})}{N\omega_N^{\frac1N}\mu(t)^{\frac{N-1}{N}}} \right)^2\int_{\rr^N}  |D u_{N}^*| \d \Per(\{u^*_N>t\})  \, \d t, \quad \forall s \in(0,\max u].
    \end{equation}
		\item 
		The Euclidean-rearrangement maps $\W_0(\Omega)$ to $\W_0(\Omega^*)$ and
	\begin{equation}\label{eq:euclpolya}
		\int_{\Omega}|{D u}|^2\d \mm\ge \Big(\frac{\tilde{C}}{N\omega_N^{1/N}}\Big)^2\int_{\Omega^*}|{D u_{N}^*}|^2\d \Leb^N, \quad \forall \, u \in \W_0(\Omega).
	\end{equation}
	\end{enumerate}		
\end{theorem}
\begin{proof}

    It is enough to prove $i)$, since $ii)$ then follows  by approximation with   Lipschitz functions using  Lemma \ref{lem:nonvanishing gradients} as in \cite[Theorem 3.6]{NV21} (see also \cite{MS20}).
    
    We fix $u \in \LIP_c(\Omega)$, $u\not\equiv 0$, with $\lip(u)\neq 0$  $\mm$-a.e.\  $\{u>0\}$. Set $M\coloneqq\sup u.$
Under these assumptions $\mu$ is strictly monotone, absolutely continuous (hence differentiable almost everywhere) and 
\begin{equation}\label{eq:polya part 2}
	\int_{\{u\le s\}} |D  u|^2\, \d \mm \ge \int_0^s  \frac{\Per(\{u>t\})^2 }{-\mu'(t)} \, \d t, \quad \forall s\in(0,M].
\end{equation}
This can be seen arguing exactly as in the proof of \cite[Prop. 3.12 and (3.23)]{MS20} (see also \cite{NV21}), recalling also that $|{\bf D}u|\le \lip(u)\mm=|Du|\mm$ (see \eqref{eq:lip pi}). Next we claim that $u^*_N \in \LIP_c(\Omega^*).$ Recall that by definition $\Omega^*=B_r(0)\subset \rr^N$, where  $r>0$ satisfies $\Leb^N(B_r(0))=\mea(\Omega).$ From the definitions $u^*_N(x)=\tilde u^*_N(|x|)$, where $\tilde u^*_N:[0,r]\to \R^+$ is the rearrangement into the space $([0,\infty),|.|,N\omega_N t^{N-1}\d t)$ as defined in \cite[Definition 3.1]{NV21}.  Then the fact that $u^*_N\in \LIP(\Omega^*)$ follows directly from  $\tilde u^*_N\in \LIP[0,r]$ which is proved in  \cite[Prop.\ 3.4]{NV21} under the same assumptions on $u$ and $\Omega$. Finally $\supp (u^*_N)\subset \Omega^*$. Indeed $\supp (u)\subsetneq \Omega$,  otherwise $\Omega$ would be closed and would coincide with $\X$ (as $\X$ is connected). This  implies that $\Leb^N(\{u_N^*>t\})\le \mea(\supp (u))<\mea(\Omega)=\Leb^N(\Omega^*),$ for all $t>0,$ because $\mea(\Omega\setminus \supp(u))>0$, as non-empty open sets in $\X$ have positive measure. Since $u_N^*$ is  a radial function centered at the origin this shows   $\supp (u^*_N)\subset \Omega$.
Next we observe that $\tilde u_{N}^*$ is strictly decreasing in $(0,\mm(\supp(u)))$ (since $\mu(t)$ is continuous) and in particular $\{u^*_{N}>t\}=B_{r_t}(0)$ (and $\{u^*_{N}=t\}=\partial B_{r_t}(0)$) for some $r_t\in[0,\mm(\Omega)]$, for every $t \in (0,M)$.  Note that $r_t$ can be computed explicitly to be $r_t=(\omega_N^{-1}\mu(t))^{1/N}$, which also shows that $(0,M)\ni t\mapsto r_t$ is a strictly monotone and locally absolutely continuous map. In particular
\begin{equation}\label{eq:surface measure}
    \hau^N(\partial B_{r_t}(0))=N \omega_N^{\frac1N} \mu(t)^{\frac{N-1}{N}}.
\end{equation}
Combining these observations with the expression for the derivative of $\mu$ given in \cite[Lemma 3.10]{MS20} (see also \cite[Lemma 3.5]{NV21}) we have 
	\begin{equation} \label{eq:formula mu'}
	    	-\mu'(t)=\int_{\partial B_{r_t}(0)} (\lip (u_{N}^*))^{-1} \d \hau^{N-1}=\frac{N \omega_N^{\frac1N} \mu(t)^{\frac{N-1}{N}}}{\lip (\tilde u_{N}^*)(r_t)}\, \quad \text{for a.e.\  $t\in (0,M)$},
	\end{equation}
where we have used \eqref{eq:surface measure} and that  $\lip (u_{N}^*)(x)=\lip (\tilde u_{N}^*)(|x|)$ which easily follows from the identity $u^*_N(x)=\tilde u^*_N(|x|)$.
Plugging the above in \eqref{eq:polya part 2} we reach
	\begin{align*}
	\int_{\{u\le s\}} |Du|^2\, \d \mm &\ge \int_0^s  \frac{\Per(\{u>t\})^2\lip (\tilde u_{N}^*)(r_t) }{N \omega_N^{\frac1N} \mu(t)^{\frac{N-1}{N}}} \, \d t \\ 
	 &=\int_0^s  \left(\frac{\Per (\{u>t\})}{N\omega_N^{\frac1N}\mu(t)^{\frac{N-1}{N}}} \right)^2 \lip (\tilde u_{N}^*)(r_t)N\omega_N^{\frac1N}\mu(t))^{\frac{N-1}{N}} \, \d t\\
	 &=\int_0^s  \left(\frac{\Per (\{u>t\})}{N\omega_N^{\frac1N}\mu(t)^{\frac{N-1}{N}}} \right)^2 \int_{\partial B_{r_t}(0)} \lip (u_{N}^*) \d \hau^{N-1}  \, \d t\,,
	\end{align*}
	where for the last step we argue as in \eqref{eq:formula mu'}. This concludes the proof.
\end{proof}

\subsection{RCD spaces}\label{sec:rcd}
For brevity we do not recall the definition of $\RCD(K,N)$ spaces  (with $N\in[1,\infty)$ and $K \in \rr$) since it will not be directly used in this note, instead  we recall here all the properties of these spaces that will be needed. For further details on the definition and on the theory of metric measure spaces with synthetic Ricci curvature lower bound we refer to the surveys \cite{AmbICM,DGmeetsG} and references therein.

First recall that every $\RCD$ space is infinitesimally Hilbertian from the very definition.
In every $\RCD(K,N)$ space  the {\emph{Bishop-Gromov inequality} holds  \cite{St2}, that is 
\begin{equation}\label{eq:BG}
    \frac{\mea(B_r(x))}{v_{K,N}(r)}\ge  \frac{\mea(B_R(x))}{v_{K,N}(R)},\quad \forall x \in \X, \, \forall 0<r<R,
\end{equation}
where the quantities $v_{K,N}(r)$ coincides, for $N\in \nn$, with the volume of the ball of radius $r$ in the model space of curvature $K$ and dimension $N$. For the definition of $v_{K,N}$ for non integer $N$ see \cite{St2}, however in the results of this note  only the case $N\in \nn$ will be relevant. 

As a consequence of \eqref{eq:BG} we obtain that for every $R_0>0$,  $N<+\infty$  and $K \in \rr$  there exists a constant $C_{R_0,K,N}$ such that for every $\RCD(K,N)$ space $\Xdm$ it holds:
\begin{equation}\label{eq:RCD doulbing}
    \frac{\mea(B_r(x))}{\mea(B_R(x))}\ge C_{R_0,K,N} \left(\frac{r}{R}\right)^N, \quad \forall x \in \X, \, \forall 0<r<R\le R_0.
\end{equation}
Taking $R=2r$ this shows also that every $\RCD(K,N)$ space with $N<+\infty$ is uniformly locally doubling (recall Definition \ref{def:PI}).
It is also proved in \cite{Rajala12,Rajala12-2} that every $\RCD(K,N)$ space supports also a weak local (1,1)-Poincar\'e inequality. Combining the last two observations we conclude that every $\RCD(K,N)$ space with $N<+\infty$ is an  infinitesimally Hilbertian PI space. 

We recall the following embedding result.
\begin{prop}[\hspace{1sp}{\cite[Theorem 6.3, $ii)$]{GMS13}}]\label{prop:embedding}
    Let $\Xdm$ be an $\rcd(K,N)$ space and let $\Omega \subset \X$ be bounded. Then the inclusion $\W_0(\Omega)\hookrightarrow L^2(\mm)$ is compact.
\end{prop}

Given an $\rcd(K,N)$ space $\Xdm$ we define the \emph{Bishop-Gromov} density function $\theta_N:\X \to (0,+\infty]$ by
\begin{equation}\label{eq:density}
    \theta_N(x)\coloneqq \lim_{r\to 0^+} \frac{\mea(B_r(x))}{\omega_N r^N}=\lim_{r\to 0^+} \frac{\mea(B_r(x))}{v_{N,K}(r)}\,,
\end{equation}
where the existence of the limits is ensured by \eqref{eq:BG} (see  \cite[Def. 1.9]{DePhGi}). As shown in \cite[Lemma 2.2]{DePhGi} the function $\theta_N$ is lower-semicontinuous in $\X.$ 

A key property that we will need is the validity of a local almost-Euclidean isoperimetric inequality. We will use the following version essentially proved in \cite{NV21} (see also \cite{CM20,AntPasPoz21} for similar results in the setting of $\rcd$ spaces and \cite{BM82} for the Riemannian setting). 
\begin{theorem}[Local almost-Euclidean isoperimetric inequality]\label{thm:local euclidean isop}
	Let $\Xdm$ be an $\rcd(K,N)$ space for some $N\in(1,\infty), K \in \rr$.  Then for every  $x \in \X$ with $\theta_N(x)<+\infty$ and every $\eps \in(0,\theta_N(x))$ there exists $\rho=\rho(\eps, x,N)$ such that
	\begin{equation}\label{eq:partial local euclidean isop}
		\Per (E) \ge  \mm(E)^\frac{N-1}{N}  N\omega_N^{\frac1N}(\theta_{N}(x)-\eps)^{\frac1N}(1-\eps), \quad \forall \, E \subset B_\rho(x) \text{ Borel.}
	\end{equation}
\end{theorem}
\begin{proof}
	It is sufficient to prove the statement with $\eps \in (0,\theta_{N}(x)/2\wedge 1/2)$. From \cite[Theorem 3.9]{NV21}  there exists $\bar R=\bar R(\eps,K,N)$ such that for every $x \in \X$, $R\in(0,\bar R]$ it holds
	\[
		\Per (E) \ge  \mm(E)^\frac{N-1}{N}  N\omega_N^{\frac1N}\theta_{N,R}(x)^{\frac1N}(1-(2C^{1/N}_{\eps,R}(x)+1)\eps-\eps), \quad \forall \, E \subset B_{\eps R}(x),
	\]
	where $\theta_{N,\rho}(x)\coloneqq\frac{\mea(B_\rho(x))}{\omega_N \rho^N}$ and $C_{\eps,R}(x)\coloneqq \frac{\theta_{N,\eps R}(x)}{\theta_{N,R}(x)}.$ Since $\theta_N(x)=\lim_{\rho \to 0}\theta_{N,\rho}(x)<+\infty$  there exists $\bar r=\bar r (x,\eps)$ so that $\theta_{N,\rho}(x)\in (\theta_{N}(x)-\eps,\theta_N(x)+\eps)$ for all $\rho\le \bar r$. Moreover since $\eps<\theta_{N}(x)/2$ we have that for every $R\le \bar r$  it holds $C_{\eps,R}(x)\le \frac{\frac{3}{2}\theta_N(x)}{\frac{1}{2}\theta_N(x)}= 3.$ Hence choosing $\rho= \rho(x,\eps,K,N)\coloneqq \eps(\bar r(\eps,x)\wedge \bar R(\eps,K,N))$, we have that 
	\[
	\Per (E) \ge  \mm(E)^\frac{N-1}{N}  N\omega_N^{\frac1N}(\theta_{N}(x)-\eps)^{\frac1N}(1-8\eps), \quad \forall \, E \subset B_\rho(x),
	\]
	from which the conclusion follows.
\end{proof}
Next we introduce the subclass  of non-collapsed $\rcd(K,N)$ space.
\begin{definition}[\hspace{1sp}\cite{DePhGi}]
    An $\rcd(K,N)$ space $\Xdm$ is said to be  \emph{non-collapsed} if $\mea=\hau^N$, where $\hau^N$ denotes the $N$-dimensional Hausdorff measure on $(\X,\sfd)$.
\end{definition}
After the works \cite{H20,BGHZ23} this definition is known to be equivalent to one given in \cite{Kit17}. 
As showed in \cite[Theorem 1.12]{DePhGi} if $\Xdm$ is a non-collapsed $\rcd(K,N)$ space, then $N\in \nn.$ Moreover by \cite[Corollary 1.7]{DePhGi} (see also \cite[Theorem 1.4]{AHT}) it holds that
\begin{equation}\label{eq:theta 1}
\begin{split}
     & \theta_N(x)=1,\quad  \mea\text{-a.e.\ $x \in \X$,} \\
     &  \theta_N(x)\le1, \quad \forall x \in \X.
\end{split}
\end{equation}

\begin{remark}[Consistency with the smooth setting]\label{rmk:rcd ok}
It is worth to recall that $\rcd$ spaces are compatible with the smooth setting in the following sense.
    Any $N$-dimensional Riemannian manifold $(M,g)$ with Ricci curvature bounded below by  a number $K\in \rr$, i.e.\ ${\rm Ric}_g\ge K g$, endowed with the Riemannian distance and volume measure is a non-collapsed $\rcd(K,N)$ metric measure space \cite{VonSt,CEMS01}. In particular the metric measure space $(\rr^N,|\cdot|,\hau^N)$, where $\hau^N$ is the $N$-dimensional Hausdorff measure is a non-collapsed $\RCD(0,N)$ space. \fr
\end{remark}

We conclude recalling the validity of the Weyl law in non-collapsed setting proved in \cite{AHT,ZZ}. 
\begin{theorem}[Weyl law in $\RCD$ spaces]\label{thm:weyl rcd}
	Let $(\X,\sfd,\hau^N)$ be an $\rcd(K,N)$ space and $\Omega \subset \X$ be open and bounded. Then
	\begin{equation}\label{eq:weyl rcd}
		\lim_{k\to +\infty} \frac{k}{\lambda_k^\cD(\Omega)^{N/2}}=\frac{\omega_N}{(2\pi)^N} \hau^N(\Omega),
	\end{equation}
 where $\{\lambda_k^\cD(\Omega)\}_{k \in \nn}$ denotes the spectrum of the Dirichlet Laplacian in $\Omega$ defined in \eqref{eq: Dirlisteigen}.
\end{theorem}
\begin{proof}
 By the results in \cite{AHT,ZZ} it holds that
    \begin{equation}\label{eq:weyl rcd-count}
		\lim_{\lambda \to +\infty} \frac{N(\lambda)}{\lambda^{N/2}}=\frac{\omega_N}{(2\pi)^N} \hau^N(\Omega),
	\end{equation}
 where $N(\lambda)\coloneqq \#\{k \in \nn \ : \lambda_k^\cD(\Omega)\le \lambda \}$. This implies \eqref{eq:weyl rcd}. To see this, set $N^-(\lambda)\coloneqq \#\{k \in \nn \ : \lambda_k^\cD(\Omega)<\lambda \}$, and observe
 \[
 N(\lambda_{k}^\cD(\Omega)-1)\le N^-(\lambda_{k}^\cD(\Omega))\le k\le N(\lambda_{k}^\cD(\Omega)), \quad \forall k \in \nn.
 \]
 Note that in \cite{ZZ} formula \eqref{eq:weyl rcd-count}  is stated with the further assumption that  $\diam(\Omega)<\diam(\X)$ (when $\Omega \neq \X$), however this assumption is needed in \cite{ZZ} only to ensure the discreteness of the spectrum, for which we know by Proposition \ref{prop:embedding} that the boundedness of $\Omega$ is sufficient.
\end{proof}

\section{Sobolev spaces and Neumann eigenfunctions in uniform domains}\label{sec:unif domain}
Our method to deal with Neumann eigenfunctions in  domains with irregular boundary  in $\rr^N$ (or in more abstract $\rcd$ spaces)   will be to translate the problem to a global one, by viewing the domain as a metric measure space. The idea is that if the boundary satisfies an appropriate   regularity condition, then the resulting m.m.\ space is a PI space and in particular has good analytic properties, like the isoperimetric inequality and embedding theorems (recall Section \ref{sec:pi}). We stress that the possibility of using abstract metric spaces to deal with Neumann and mixed boundary value problems in irregular domains of the Euclidean space was noted before  (see e.g.\ \cite[Page 33]{BB13}). 

The key notion that we will use is the one of \emph{uniform domain} that we now introduce.
\begin{definition}[Uniform domains]\label{def: unif dom}
     A bounded open subset $\Omega$ of a metric space $(\X,\sfd)$ is called a \emph{uniform domain} if there exists a constant $C>1$ such that every pair of points $x,y \in \Omega$ can be joined by a rectifiable curve  $\gamma:[0,1]\to \Omega$ such that $l(\gamma)\le C\sfd(x,y)$ and 
    \[
    \sfd(\gamma(t),\partial \Omega)\ge C^{-1} \min\left(l(\gamma\restr{[0,t]}),l(\gamma\restr{[t,1]})\right), \quad \forall t\in[0,1].
    \]
\end{definition}
Uniform domains were introduced in \cite{MartioSarvas} and \cite{Jones81} (see also \cite{Va88}) and are  central in the theory of BV and Sobolev extension domains (see \cite{Jones81,GehOs,HK91,HK92,BjSh07,Lahti}).
Uniform domains are also equivalent to one-sided non tangentially accessible (1-sided NTA) domains 
(see e.g.\ \cite[Theorem 2.15]{AHMNT17} or the Appendix in \cite{nta}).
They include Lipschitz domains, but also more irregular domains such as the quasi-disks, i.e.\  images of the unit ball under a global quasi-conformal maps  (see \cite[Theorem 2.15]{MartioSarvas}, \cite{GehOs}, \cite[Section 3]{JK82} or also \cite[Remarks 2.1]{ToroSurvey}). In particular the  interior of the Koch snowflake is an example of  uniform domain. It has also been proved recently that in every doubling quasi-convex metric space (in particular any bounded $\rcd(K,N)$ space) any bounded open set can be approximated from inside and outside by uniform domains, see \cite{Raj20} for the precise statement.  

It easily follows from the definition  that every uniform domain is both connected and locally connected (see e.g.\ \cite{Martio}).

The main goal of this section is to prove the following theorem. It says that a uniform domain in a non-collapsed $\RCD(K,N)$ space (and in particular in $\rr^N$) when viewed as a m.m.\ space admits an almost Euclidean isoperimetric inequality near almost every point  (point $\ref{it:isop}$) and the eigenvalues of the Laplacian satisfy a weak version of the Weyl law  (point $\ref{it:weyl}$). Item $\ref{it:techdim}$ is a technical condition that we will need to apply the Faber-Krahn inequality in Proposition \ref{prop:Faber Kran PI}. We remark that, differently from the preceding sections, in the following statement we will denote by $\Y$ the ambient space, while the notation $\Xdm$ is reserved to the metric measure space associated to the closure of the domain $\Omega\subset \Y$.  
\begin{theorem}\label{th:almost-iso}
    Let  $(\Y,\tilde{\sfd},\hau^N)$ be an $\RCD(K,N)$ space and let $\Omega\subset \Y$ be a uniform domain. Then the metric measure space $\Xdm\coloneqq (\overline\Omega,\tilde{\sfd}\restr{\overline \Omega}, \hau^N\restr {\overline\Omega})$ is an infinitesimally Hilbertian PI space and satisfies the following properties:
   \begin{enumerate}[label=(\roman*)]
   \item\label{it:techdim} there exists a constant $c>0$ such that
	\[
	\frac{\mea(B^{\X}_r(x))}{\mea(B^{\X}_R(x))}\ge c\left(\frac{r}{R}\right)^N, \quad \forall\, x \in \X,\, \forall\, 0<r<R.
	\]
       \item\label{it:isop} For every $\eps>0$ there exists a closed set $C_\eps\subset \X$ with $\mea(C_\eps)=0$ such that for every $x \in \X\setminus C_\eps$ there exists a constant $\rho=\rho(x,N,\eps)>0$  satisfying
    \begin{equation}\label{eq:local eucl}
		    \Per_\X (E)\ge(1-\eps)N\omega_N^{\frac1N}\mea(E)^{\frac{N-1}{N}}, \quad \forall \, E \subset B^\X_{\rho}(x)\, \text{ Borel,}
		\end{equation}
  where  $ \Per_\X$ and $B^\X_{\rho}(x)$ are respectively the perimeter and the metric ball in the space $\Xdm.$
  \item\label{it:weyl} denoted by $\{\lambda_k\}_k$  the spectrum of the Laplacian in $\Xdm$ (recall \eqref{eq:PI discrete sepctrum}) it holds $\lambda_k=\lambda_k^{\cN}(\Omega)$ for all $k\in \nn$ and
  \begin{equation}\label{eq:weyl}
		\limsup_{k\to +\infty} \frac{\lambda_k^{N/2}}{k}\le\frac{(2\pi)^N }{\omega_N\mea(\X)}.
	\end{equation}
   \end{enumerate}
\end{theorem}

\begin{remark}[The `bad' set $C_\eps$]\label{rmk:bad set}
If the ambient space $(\Y,\tilde{\sfd},\hau^N)$ is the Euclidean space then item $ii)$ in Theorem \ref{th:almost-iso} is immediate by taking $C_\eps=\partial \Omega$ and by the isoperimetric inequality (indeed $\partial \Omega$ is negligible as we will show in Lemma \ref{lem:zero boundary}). More generally if $(\Y,\tilde{\sfd},\hau^N)$ is a Riemannian manifold with $\tilde{\sfd}$ the geodesic distance then item $ii)$ follows taking again $C_\eps=\partial \Omega$ and applying the local almost-Euclidean isoperimetric inequality in \cite[Appendice C]{BM82}. 
In the general case, as will be shown in the proof,  the set $C_\eps$ in Theorem \ref{th:almost-iso} can be taken to be 
  $$C_\eps\coloneqq \{x \in \overline\Omega \ : \ \theta_{N}(x)\le 1-\eps\}\cup \partial \Omega,$$
  where $\theta_N:\X\to (0,+\infty]$ is the Bishop-Gromov density function defined in \eqref{eq:density}.  In other words the set $C_\eps$ contains $\partial \Omega$ plus a subset of the singular points, the latter being the points where $\theta_{N}< 1$. Note that $\theta_N\equiv 1$ if $(\Y,\tilde{\sfd},\hau^N)$ is a Riemannian manifold, so that $C_\eps$ reduces to $\partial \Omega,$ in accordance to what we said above. \fr
\end{remark}

We start with some basic properties of uniform domains.
\begin{definition}[Corkscrew-condition]\label{def:cork}
    A bounded open subset $\Omega$ of a metric space $(\X,\sfd)$ satisfies the corkscrew-condition if there exists a constant $\eps>0$ such that 
    for every point $x \in \overline \Omega$ and all $0<r\le \diam(\Omega)$, the set $\Omega\cap B_r(x)$ contains a ball of radius $\eps r.$
\end{definition}
In Definition \ref{def:cork} it is equivalent to require only $x \in \partial \Omega.$ 

For the proof of the following well known fact see e.g.\ \cite[Lemma 4.2]{BjSh07}.
\begin{lemma}\label{lem:cork unif}
    Every uniform domain  satisfies the corkscrew-condition.
\end{lemma}

\begin{lemma}\label{lem:zero boundary}
    Let $\Xdm$ be a uniformly locally doubling metric measure space and let $\Omega\subset \X$ satisfy the corkscrew-condition. Then $\mea(\partial \Omega)=0.$ In particular this holds if $\Omega\subset \X$ is a uniform domain.
\end{lemma}
\begin{proof}
    Thanks to the  corkscrew-condition there exists a constant $\eps>0$ such that for every $x \in \partial \Omega$ and $r>0$ there exists $B_{\eps r}(y)\subset B_r(x)\cap \Omega$. Then by the uniformly locally doubling assumption we have 
\begin{equation}\label{eq:cork bound}
    \mea(B_r(x)\cap \Omega)\ge \mea(B_{\eps r}(y))\ge C_\eps \mea(B_{2r}(y))\ge  C_\eps\mea(B_r(x)), \quad \forall r\in(0,1),
\end{equation}
where $C_\eps>0$ is a constant depending only on $\eps$. Therefore no point of $x \in \partial \Omega$ can be a one-density point for $\partial \Omega$. The conclusion follows by the Lebesgue differentiation theorem for locally doubling metric measure spaces (see e.g. \cite[Section 3.4]{HKST}). The result applies to uniform domains, since by Lemma \ref{lem:cork unif} they satisfy they corkscrew-condition.
\end{proof}
The following lemma gives a lower bound on the measure of balls in domain satisfying the corkscrew-condition.
\begin{lemma}\label{lem:dim doubling}
Fix $s>0$. Let $\Xdm$ be a metric measure space such that for every $R_0>0$ there exists a constant $c_0>0$ such that
\begin{equation} \label{eq:dim doubling}
     \frac{\mm(B_r(x))}{\mm(B_R(x))}\ge c_0\left(\frac{r}{R}\right)^s, \quad \forall x \in \X, \,\, \forall 0<r<R\le R_0.
\end{equation}
Then for every bounded domain $\Omega\subset \X$ satisfying the corkscrew-condition  there exists a constant $C>0$  such that
\begin{equation}\label{eq:cork doubling}
    \frac{\mm(B_r(x)\cap \Omega)}{\mm(B_R(x)\cap \Omega)}\ge C\left(\frac{r}{R}\right)^s, \quad \forall x \in \overline \Omega, \,\, \forall 0<r<R.
\end{equation}
In particular this holds if $\Omega\subset \X$ is a uniform domain.
\end{lemma}
\begin{proof}
Taking $R=2r$ in \eqref{eq:dim doubling} shows  that $\Xdm$ is uniformly locally doubling.  Then as in \eqref{eq:cork bound} we have the existence of a constant $\tilde C$ such that for every $x\in \overline \Omega$
    \begin{align*}
     &\mea(B_r(x)\cap \Omega)\ge \tilde C\mea(B_r(x)) \overset{\eqref{eq:dim doubling}}{ \ge} c_0\tilde C\mea(B_R(x)) \left(\frac{r}{R}\right)^s\\
&\ge  c_0\tilde C\mea(B_R(x)\cap \Omega) \left(\frac{r}{R}\right)^s,\quad \forall 0<r< R\le  \diam(\Omega).
    \end{align*}
    This proves \eqref{eq:cork doubling} for $0<r<R\le \diam(\Omega)$. 
     If instead $r> \diam(\Omega)$ we have
     \[
      \frac{\mm(B_r(x))\cap \Omega}{\mm(B_R(x)\cap \Omega)}= \frac{\mm(\Omega)}{\mm(\Omega)} =1\ge \frac{r^s}{R^s}, \quad \forall x \in \overline \Omega, \quad \forall R\ge r.
     \]
     This shows \eqref{eq:cork doubling} also for $\diam(\Omega)<r<R$ and concludes the proof.
\end{proof}

The next step is to show that the Sobolev space on a uniform domain coincides with the Sobolev space on its closure. In what follows, given a m.m.\ space $\Xdm$ and an open subset $\Omega \subset \X$, we denote by $(\overline \Omega,\sfd\restr {\overline \Omega }, \mmomega)$ the m.m.\ space obtained by endowing $\overline \Omega$ with the restriction distance $\sfdomega \coloneqq \sfd\restr{\overline \Omega \times \overline \Omega}$ and measure $\mmomega$, obtained by restricting $\mm$ to the induced Borel $\sigma$-algebra on $\overline \Omega.$ Note that by definition $(\overline \Omega,\sfd\restr {\overline \Omega }, \mmomega)$ is a complete and separable metric measure space with $\supp(\mmomega)=\overline \Omega.$ Moreover for a function $u \in L^2(\overline \Omega,\mmomega) $ we will denote by $u\restr\Omega\in L^2(\Omega,\mm)$ the function which agrees $\mm$-a.e.\ with $u$ in $\Omega.$

\begin{theorem}[Equivalence between $\W(\overline \Omega,\sfd\restr {\overline \Omega }, \mmomega)$ and $\W(\Omega)$]\label{thm:sobolev compatibility}
     Let $\Xdm$ be a PI space and $\Omega\subset \X$ be a uniform domain. Then $\W(\overline \Omega,\sfdomega, \mmomega)=\W(\Omega)$ as function spaces. More precisely, for every $u \in \W(\overline \Omega,\sfd\restr {\overline \Omega }, \mmomega)$ it holds that $\Phi(u)\coloneqq u\restr \Omega \in \W(\Omega)$  and the map
      $\Phi: \W(\overline \Omega,\sfd\restr {\overline \Omega }, \mmomega)\to \W(\Omega)$ is a surjective isometry. Moreover if $\Xdm$ is infinitesimally Hilbertian, so is $(\overline \Omega,\sfd\restr {\overline \Omega }, \mmomega)$  and 
      \begin{equation}\label{eq:scalar compatibility}
          \int_{\overline \Omega} \nabla u\cdot \nabla v \,\d \mm= \int_{\Omega} \nabla (u\restr\Omega) \cdot \nabla (v\restr\Omega) \,\d \mm,\quad \forall u,v \in  \W(\overline \Omega,\sfd\restr {\overline \Omega }, \mmomega).
      \end{equation}
\end{theorem}
\begin{proof}
The  result is well known, but for convenience of the reader we include a short argument which is  a combination of results already present in literature.  Fix $u \in\W(\overline \Omega,\sfd\restr {\overline \Omega }, \mmomega) $. The fact that $\Phi(u)=u\restr \Omega \in \W(\Omega)$ is checked e.g.\ in  \cite[Remark 2.15]{AH18}. Moreover by \cite[Prop.\ 6.4]{AMS16}   for every $\eta \in \LIP_{bs}(\Omega)$, since $\sfd(\supp(\eta u), \X\setminus \Omega)>0,$ it holds that $\eta u \in \W(\X)$ (when extended by zero in the whole $\X$) and also that $|D (\eta u)|_\X=|D(\eta u)|_{\overline\Omega}$ $\mea$-a.e.\ in $\Omega$, where $|D(\cdot)|_{\overline\Omega}$ is the w.u.g.\ in $\W(\overline \Omega,\sfdomega, \mmomega)$ and $|D(\cdot)|_{\X}$ is the w.u.g.\ in $\W(\X)$. Hence by the arbitrariness of $\eta$ and by  locality we deduce that $|D u|_{\overline\Omega}=|D(u\restr \Omega)|_{ \Omega}$ $\mea$-a.e.\ in $\Omega$, where  $|D(\cdot)|_{\Omega}$ is the w.u.g.\ in $\W(\Omega)$ (as defined in \eqref{eq:wug local}). This together with the fact that $\mea(\partial \Omega)=0$ (recall  Lemma \ref{lem:zero boundary}) shows  that $\Phi(u)$ preserves the norm.
It remains to show that $\Phi$ is surjective (note that everything we said up to now holds for an arbitrary open set $\Omega$ with $\mea(\partial \Omega)$=0).
Now since $\Omega$ is a uniform domain, by \cite[Prop. 5.9]{BjSh07}, $\Omega$ is an extension domain and in particular for every $u\in \W(\Omega)$ there exists $\tilde u \in \W(\X)$ such that $\tilde u\restr \Omega=u$. By density of Lipschitz functions (see \cite{Shanmugalingam00} or \cite[Theorem 5.1]{BB13}) there exists a sequence $u_n \in \LIP_{bs}(\X)$  such that $u_n \to \tilde u$ in $\W(\X).$ In particular it holds $u_n\restr \Omega \to  u $ in $L^2(\overline\Omega;\mmomega)$. Moreover $u_n \restr{\overline\Omega}\in \LIP(\overline\Omega,\sfd\restr{\overline{\Omega}})$ and by definition of slope  we have $\lip_{\overline\Omega}(u_n\restr{\overline\Omega})(x)\le \lip (u_n)(x)$ for every $x \in \overline \Omega$, where $\lip_{\overline\Omega}$ denotes the slope with respect to the metric space $(\overline\Omega,\sfdomega).$ Therefore 
$\sup_n \int_{\overline \Omega} |\lip_{\overline\Omega }(u_n\restr{\overline\Omega})|^2\d  \mm \le \sup_n \int_{\X} |\lip (u_n)|^2  \d \mm<+\infty$, which by definition proves that $\tilde u\restr{\overline\Omega} \in\W(\overline \Omega,\sfdomega, \mmomega). $ Since $\tilde u\restr\Omega =u$ and  by the arbitrariness of $u\in \W(\Omega)$ this shows that $\Phi$ is surjective. Finally if $\Xdm$ is infinitesimally Hilbertian then $\W(\Omega)$ is a Hilbert space and since $\Phi$ is an isometry, then $\W(\overline \Omega,\sfdomega, \mmomega)$ is a Hilbert space as well and so $(\overline \Omega,\sfdomega, \mmomega)$ is infinitesimally Hilbertian. Then  identity \eqref{eq:scalar compatibility} follows by polarization using the definition of scalar product between gradients.
\end{proof}

\begin{theorem}\label{thm:uniform implies PI}
   Let $\Xdm$ be an infinitesimally Hilbertian PI space and let $\Omega\subset \X$ be a uniform domain. Then the metric measure space $(\overline \Omega,\sfd\restr {\overline \Omega }, \mmomega)$ is an infinitesimally Hilbertian PI space.
\end{theorem}
\begin{proof}
    The fact that $(\overline \Omega,\sfd\restr {\overline \Omega }, \mmomega)$ is a PI space is proved in  \cite[Theorem 4.4]{BjSh07} (see also \cite[Prop. 7.1]{AikShan05} and \cite[Theorem A.21]{BB13}), recalling also that $\mea(\partial \Omega)=0$. The infinitesimal Hilbertianity follows   from Theorem \ref{thm:sobolev compatibility} (alternatively we could apply \cite[Prop. 4.22]{Gigli12}).
\end{proof}

From the previous results about the compatibility of the two  Sobolev spaces   $\W(\Omega)$ and $\W(\overline \Omega,\sfd\restr {\overline \Omega }, \mmomega)$, we can deduce a compatibility between eigenfunctions and eigenvalues in $\Omega$ and in $(\overline \Omega,\sfd\restr {\overline \Omega }, \mmomega)$.
\begin{cor}[Equivalence between $\lambda_k^\cN(\Omega)$ and $\lambda_k(\overline \Omega,\sfdomega, \mmomega)$]\label{cor:spectrum compatilbity}
  Let $\Xdm$ be an infinitesimally Hilbertian PI space and let $\Omega\subset \X$ be a uniform domain. Then:
  \begin{enumerate}[label=\roman*)]
    \item $u \in \W(\Omega)$ is an eigenfunction for the Neumann Laplacian in $\Omega$ of eigenvalue $\lambda $ if and only  there exists an eigenfunction  $\tilde u\in \W(\overline \Omega,\sfd\restr {\overline \Omega }, \mmomega)$ of the Laplacian in the metric measure space $(\overline \Omega,\sfdomega, \mmomega)$ of eigenvalue $\lambda$ and satisfying  $\tilde u\restr\Omega =u$. In particular, if this is the case, then $u$ has a  H\"older continuous representative in $\Omega.$
      \item  
  the embedding $\W(\Omega)\hookrightarrow L^2(\Omega)$ is compact and in particular the Neumann Laplacian in $\Omega$ has a discrete spectrum $\{\lambda_k^\cN(\Omega)\}_{k \in \nn}$ (counted with multiplicity)  satisfying
  \begin{equation*}
  0=\lambda_1^\cN(\Omega)\le \lambda_2^\cN(\Omega)\le \dots \lambda_k^\cN(\Omega)\le \dots \rightarrow +\infty.
  \end{equation*}

  \item denoted by $\{\lambda_k\}_k$ the spectrum for the Laplacian in $ (\overline \Omega,\sfdomega, \mmomega),$  it holds that 
  \[
  \lambda_k^\cN(\Omega)=\lambda_k, \quad \forall k \in \nn.
  \]
  \end{enumerate}
\end{cor}
\begin{proof}
    First observe that the statement makes sense since by Theorem \ref{thm:uniform implies PI} we have that $(\overline \Omega,\sfdomega, \mmomega)$ is a bounded  infinitesimally Hilbertian PI space and so by \eqref{eq:PI discrete sepctrum} we have that the Laplacian in $ (\overline \Omega,\sfdomega, \mmomega)$ has a discrete spectrum $\{\lambda_k\}_{k}.$ It is sufficient to show $i)$ and $ii)$, because then  $iii)$ would follow from the definitions. Suppose that $u \in \W(\Omega)$ is an eigenfunction for the Neumann Laplacian in $\Omega$ of eigenvalue $\lambda $. Then by Theorem \ref{thm:sobolev compatibility}  there exists $\tilde u \in \W(\overline \Omega,\sfdomega, \mmomega)$ such that $\tilde u \restr\Omega=u.$
    Moreover for every $v \in \W(\overline \Omega,\sfdomega, \mmomega)$, again by Theorem \ref{thm:sobolev compatibility}, we have that $v\restr\Omega \in \W(\Omega)$. Then applying \eqref{eq:scalar compatibility} we obtain
    \[
     \int_{\overline \Omega} \nabla \tilde u\cdot \nabla v \,\d \mm= \int_{\Omega} \nabla u \cdot \nabla v\restr\Omega \,\d \mm=-\lambda \int_\Omega uv\restr\Omega \d \mm=-\lambda \int_{\overline \Omega} \tilde uv \d \mm,
    \]
where in the second identity we used the definition of eigenfunction and in the last one that $\mea(\partial \Omega)=0,$ because $\Omega$ is a uniform domain. This shows that $\tilde u$ is an eigenfunction of eigenvalue $\lambda$ for the Laplacian in $(\overline \Omega,\sfdomega, \mmomega)$. Then by Theorem \ref{thm:continuous eigen} $\tilde u$ has a H\"older continuous representative in $\overline \Omega$, which implies that $u$ has also a continuous representative in $\Omega$. Conversely suppose that $\tilde u$  is an eigenfunction of eigenvalue $\lambda$ for the Laplacian in $(\overline \Omega,\sfdomega, \mmomega)$. Then by Theorem \ref{thm:sobolev compatibility} we have $\tilde u \restr \Omega \in \W(\Omega)$. Moreover, again by Theorem \ref{thm:sobolev compatibility}, for every $v \in \W(\Omega)$ there exists $\tilde v \in  \W(\overline \Omega,\sfdomega, \mmomega)$ such that $\tilde v \restr\Omega=v.$ Therefore as above using  \eqref{eq:scalar compatibility}
 \[
    -\lambda \int_\Omega \tilde u\restr\Omega v\d \mm=-\lambda \int_{\overline \Omega} \tilde u \tilde v \d \mm= \int_{\overline \Omega} \nabla \tilde u\cdot \nabla \tilde v \,\d \mm= \int_{\Omega} \nabla \tilde u\restr\Omega \cdot \nabla v \,\d \mm.
    \]
  This shows that $\tilde u\restr \Omega$ is an eigenfunction of eigenvalue $\lambda$ for the Neumann Laplacian in $\Omega$ and completes the proof of $i).$ For $ii)$ recall that by Theorem \ref{thm:sobolev compatibility} the map  $\Phi:  \W(\overline \Omega,\sfdomega, \mmomega)\to \W(\Omega)$, given by $\Phi(u)=u\restr \Omega$ is an isometry and that by Theorem \ref{thm:compact embedding} the inclusion $\iota: \W(\overline \Omega,\sfdomega, \mmomega)\hookrightarrow L^2(\overline{\Omega},\mmomega)$ is compact. Let now $u_n \in \W(\Omega)$ be a sequence bounded in $\W(\Omega).$ Then the sequence $\Phi^{-1}(u_n)\in \W(\overline \Omega,\sfdomega, \mmomega)$ is also bounded, hence it has a converging subsequence in $L^2(\overline{\Omega},\mmomega)$. However by definition $\Phi^{-1}(u_n)\restr \Omega=\Phi(\Phi^{-1}(u_n)=u_n$ for every $n$. Hence $u_n$ has also a converging subsequence in $L^2(\Omega,\mm\restr\Omega)$, which shows that the embedding $\W(\Omega)\hookrightarrow L^2(\Omega,\mm\restr\Omega)$ is compact. This completes the proof of $ii).$
    \end{proof}

\begin{remark}
  If the ambient space $\Xdm$ is an $\rcd(K,N)$ space, with $N<+\infty$, then the Neumann and Dirichlet eigenfunctions are actually locally Lipschitz in the interior of the domain, as follows directly from \cite[Theoem 1.1]{Jiang14} (see also \cite[Prop. 7.1]{AHPT21}). Recall that the continuity of eigenfunctions  is crucial to define their nodal domains.
  \fr  
\end{remark}

Next we show that for a uniform domain $\Omega $ there is a one to one correspondence between the nodal domains in $\Omega$ and the nodal domains in its closure.
\begin{prop}[Compatibility of nodal domains]\label{prop:components compatibility}
    Let $(\X,\sfd)$ be a metric space, $\Omega\subset \X$ be a uniform domain and $f:\overline \Omega\to \rr$ be a continuous function. Denote by $\mathcal C$ (resp.\ $\overline{\mathcal C}$) the set all the connected components of   $\Omega  \setminus \{f=0\}$ (resp.\   $\overline \Omega  \setminus \{f=0\}$). Then
    \begin{equation}\label{eq:components compatibility}
         \overline  {\mathcal C}=\{U\cup (\partial U\cap (\partial \Omega\setminus\{f=0\})) \ : \ U\in \mathcal C\}. 
    \end{equation}
    In particular the sets $\mathcal C$ and $\overline{\mathcal C}$ have the same cardinality.
\end{prop}
\begin{proof}
For every $U\in \mathcal{C}$ we put $\phi(U)\coloneqq U\cup (\partial U\cap (\partial \Omega\setminus\{f=0\}))$ and we want to show that $\phi$ defines a bijective map $\phi: \mathcal C\to \overline{\mathcal C}$. We note immediately that $\phi$ is injective, because if $\phi(U)=\phi(V)$  then $U=\phi(U)\cap \Omega=\phi(V)\cap \Omega=V $. 

 To conclude it is sufficient to prove that the sets  $\{\phi(U)\}_{U \in \mathcal C}$   are open, closed and connected in the topology of $\overline \Omega  \setminus \{f=0\}$ and  that their union is $\overline \Omega  \setminus \{f=0\}$. Indeed this would show that  $\{\phi(U)\}_{U \in \mathcal C}$   are exactly the connected components of $\overline \Omega  \setminus \{f=0\}$, which would imply \eqref{eq:components compatibility} and in particular that $\phi$ is surjective.

Note first that the elements of  $\mathcal C$ are open in the topology of $\Omega \setminus \{f=0\} $ (and thus also in the one of $\X$) because $\Omega \setminus \{f=0\} $ is locally connected   (recall Lemma \ref{lem:basic nodal}) since it is an open subset of $\Omega$ which is a uniform domain and thus locally connected.

The key observation is that for every $x \in \partial \Omega\setminus\{f=0\}$ there exists $r>0$ and $U \in \mathcal C$ such that
\begin{equation}\label{eq:key}
B_r(x)\cap \overline \Omega\subset \phi(U)= U\cup (\partial U\cap (\partial \Omega\setminus\{f=0\})).
\end{equation}
 To prove this note that $f(x)\neq 0$.  Then by continuity there exists $r_0>0$ so that   $B_r(x)\cap  \overline  \Omega\subset\overline \Omega \setminus  \{f =0\}$ for every $r\in (0,r_0]$. In particular for every $r\in (0,r_0]$ the set $B_r(x)\cap  \Omega$ is contained in $\bigcup_{U \in \mathcal C} U$. Suppose that for some $r>0$ the set $B_r(x)\cap  \Omega$ intersects at at least two distinct sets $U_r,V_r\in \mathcal C$ and take two points $u_r\in U_r\cap B_r(x),v_r\in V_r\cap B_r(x)$. Since $\Omega$ is a uniform domain there exists a rectifiable curve $\gamma:[0,1]\to \X$ contained in $\Omega$, connecting $u_r$ and $v_r$ and of length $l(\gamma)\le C\sfd(u_r,v_r)\le 2Cr$, where $C>0$ is some constant independent of $r.$ Then $\gamma([0,1])$ must  intersect $\{f=0\}$ otherwise $\gamma([0,1])\subset \Omega \setminus\{f=0\}$ and $U_r\cup V_r\cup \gamma([0,1])$ would be a connected subset of $\Omega\setminus\{f=0\}$, which contradicts the fact that $U_r,V_r$ are distinct connected components of  $\Omega \setminus\{f=0\}.$  Therefore, since $\gamma([0,1]) \subset B_{(2C+1)r}(x)$, it holds $ B_{(2C+1)r}(x)\cap \Omega \cap  \{f =0\}\neq \emptyset$ and so by the choice of $r_0$ we must have $(2C+1)r>r_0$. This proves  that for $r$ small enough $B_r(x)\cap  \Omega$ intersects at most one set in $\mathcal C$. However, as observed above, $B_r(x)\cap  \Omega\subset \bigcup_{U \in \mathcal C} U$ for every $r \in (0,r_0]$. Therefore for $r>0$ small enough we must have that $ B_r(x)\cap \Omega \subset U$ for some $U \in \mathcal C.$ Fix one such $r>0$ and fix $y \in  B_r(x)\cap \partial \Omega$. Then for every $s>0$ small enough $B_s(y)\cap \Omega\neq \emptyset$ and $B_s(y)\cap \Omega\subset B_r(x)\cap \Omega\subset U$. This shows that $y \in \partial U$. Therefore  for every $r>0$ small enough $B_r(x)\cap \Omega \subset U$, $B_r(x)\cap \partial \Omega \subset  \partial U$ for some $U\in \mathcal C$ and as observed above $B_r(x)\cap \overline \Omega\subset \overline \Omega \setminus\{f=0\}$. Combining these three facts  proves \eqref{eq:key}.
 
Consider now any $U \in \mathcal C$ and note that
       $\phi(U)$ is precisely the closure of $U$ in the topology of $\overline \Omega  \setminus \{f=0\}$. 
    Indeed
    \begin{align*}
        \overline{U}\cap( \overline \Omega  \setminus \{f=0\})&
        =\left(U\cap ( \overline \Omega  \setminus \{f=0\}) \right)\cup 
        \left(\partial U\cap ( \overline \Omega  \setminus \{f=0\}) \right)\\
        &= U \cup \left(\partial U\cap ( \partial \Omega  \setminus \{f=0\}) \right)\cup \left(\partial U\cap (\Omega  \setminus \{f=0\}) \right)\\
        &=U \cup \left(\partial U\cap ( \partial \Omega  \setminus \{f=0\}) \right)=\phi(U),
    \end{align*}
    where we used that $U\subset \Omega  \setminus \{f=0\}$ and that  $\partial U\cap (\Omega  \setminus \{f=0\})=\emptyset$ because $U$ is  closed in the topology of $\Omega  \setminus \{f=0\}$, being a connected component, but also open  as observed above. Moreover $U$ is connected in $\Omega  \setminus \{f=0\}$ and thus also in $\overline \Omega  \setminus \{f=0\}$, hence  $\phi(U)$ is also connected in the same topology, being the closure of a connected set. Additionally \eqref{eq:key} implies that $\phi(U)$ is  open in the topology of $\overline \Omega \setminus \{f=0\}$. It remains to prove that the union of the sets $\{\phi(U)\}_{U \in \mathcal C}$ is $\overline \Omega  \setminus \{f=0\}$, which can be seen as follows
       \begin{align*}
     \overline \Omega \setminus \{f=0\} &=(\Omega \setminus \{f=0\})  \cup  
     (\partial \Omega \setminus \{f=0\}) \overset{\eqref{eq:key}}{\subset} (\Omega \setminus \{f=0\})  \cup \left(\bigcup_{U\in \mathcal C} \phi(U)\right) \\
     &\subset \left(\bigcup_{U\in \mathcal C,} U \right)\cup  \left(\bigcup_{U\in \mathcal C} \phi(U)\right) 
     =\bigcup_{U\in \mathcal C} \phi(U).
\end{align*}

\end{proof}

We pass now to prove the main result of this section.

\begin{proof}[Proof of Theorem \ref{th:almost-iso}]
The fact that $\Xdm$ is an infinitesimally Hilbertian PI space is contained in Theorem \ref{thm:uniform implies PI}. Recall also that by  Lemma \ref{lem:zero boundary} it holds $\hau^N(\partial \Omega)=0.$  Item $\ref{it:techdim}$ follows immediately combining \eqref{eq:RCD doulbing} and Lemma \ref{lem:dim doubling} and recalling that $\Omega$ is bounded. 
   
   We pass to the proof of $\ref{it:isop}.$  Fix $\eps\in(0,1)$ arbitrary. 
We choose
$$C_\eps\coloneqq \{x \in \overline\Omega \ : \ \theta_{N}(x)\le 1-\eps\}\cup \partial \Omega$$ (see \eqref{eq:density} for the definition of $\theta_N$). From the lower semicontinuity of the function $ \theta_N(\cdot)$, it follows that $C_\eps$ is a closed subset of $\overline\Omega$. Moreover, since $\theta_{N}(x)=1$ for $\hau^N$-a.e.\ $x$ (recall \eqref{eq:theta 1}) and $\hau^N(\partial \Omega)=0,$ it follows  $\hau^N(C_\eps)=0$. Note that by construction $1-\eps < \theta_N(x)\le 1$ for all $x \in \overline\Omega \setminus C_\eps\subset \Omega$. Therefore we can apply the local almost-Euclidean isoperimetric inequality given by  Theorem \ref{thm:local euclidean isop} (recall also Lemma \ref{lem:perimeter inside}) and obtain that for every $x \in \overline\Omega \setminus C_\eps$ and every $\eps \in(0,1/4)$ there exists $\rho=\rho(x,N,\eps)<\tilde \sfd(x,\Omega^c)$ such that
	\[
		\Per_\X( E) \ge  \mm(E)^\frac{N-1}{N}  N\omega_N^{\frac1N}(1-2\eps)^{\frac1N}(1-\eps), \quad \forall \, E \subset B_\rho(x)=B^\X_\rho(x) \text{ Borel},
	\]	
 which shows \eqref{eq:local eucl}. It remains to show $\ref{it:weyl}$. Denote by $\lambda_k$, $k \in \nn$ the spectrum of the Laplacian in $\Xdm$ (counted with multiplicity and in non-decreasing order). Recall that by Corollary \ref{cor:spectrum compatilbity} $\lambda_k=\lambda_k^\cN(\Omega)$ for every $k \in \nn$, where $\lambda_k^\cN(\Omega)$ is the $k$-th Neumann eigenvalue of $\Omega$ (in non-decreasing order).  Recalling Lemma \ref{lem:discrete spectrum} we know that $\lambda_k^\cN(\Omega)\le \lambda_k^\cD(\Omega)$ for every $k \in \nn$, where $\lambda_k^\cD(\Omega)$ is the $k$-th Dirichlet Laplacian eigenvalue of $\Omega.$ Then \eqref{eq:weyl} follows from the Weyl law for the  Dirichlet Laplacian (see Theorem \ref{thm:weyl rcd}):
     \[
     \limsup_{k\to +\infty} \frac{\lambda_k^{N/2}}{k}=\limsup_{k\to +\infty} \frac{(\lambda_k^{\mathcal N}(\Omega))^{N/2}}{k}\le\lim_{k\to +\infty} \frac{(\lambda_k^\cD(\Omega))^{N/2}}{k}=\frac{(2\pi)^N }{\omega_n\hau^N(\Omega)}=\frac{(2\pi)^N }{\omega_n\mea(\X)},
     \]
     having used again $\mea(\partial \Omega)=0.$
\end{proof}

\section{From local to global isoperimetric inequality}\label{sec:isop}
In this section we prove the following crucial result. Informally speaking, it says that in a PI space satisfying an almost-Euclidean isoperimetric inequality around almost-every point, the same isoperimetric inequality extends to all sets having sufficiently small volume and avoiding a  `bad' but small region of the space.
\begin{theorem}\label{thm:almost euclidean for small volumes PI}
    Let $\Xdm$ be a bounded PI space and fix $N>1$. Suppose that for every $\eps>0$ there exists a closed set $C_\eps\subset \X$ with $\mea(C_\eps)=0$ such that for every $x \in \X\setminus C_\eps$ there exists a constant $\rho=\rho(x,N,\eps)>0$  satisfying
    \begin{equation}\label{eq:local eucl pi}
		    \Per(E)\ge(1-\eps)N\omega_N^{\frac1N}\mea(E)^{\frac{N-1}{N}}, \quad \forall \, E \subset B_{\rho}(x)\, \text{ Borel.}
		\end{equation}
    Then for every $\eps\in(0,1)$ and $\eta>0$ there exists an open set $U_{\eps,\eta}\subset \X$ with $\mm(U_{\eps,\eta})<\eta$ and  constants $\beta=\beta(\X,\eps, N,\eta)>0$, $\beta'=\beta'(\eps)>0$ such that
		\begin{equation}\label{eq:euclidean isop}
			\Per(E)\ge (1-\eps)N\omega_N^{\frac1N}\mea(E)^{\frac{N-1}{N}},
		\end{equation}
	for every $E\subset \X$ Borel  satisfying
	\[
	0<\mea(E)\le \beta, \quad \frac{\mea(E\cap U_{\eps,\eta})}{\mea(E)}\le \beta'.
	\]
\end{theorem}
Observe that assumption \eqref{eq:local eucl pi} is the same as item $ii)$ in Theorem \ref{th:almost-iso}.

The proof of Theorem \ref{thm:almost euclidean for small volumes PI} takes inspiration from the arguments in Appendix $C$ of \cite{BM82} in the smooth setting,   but requires also to deal with the technical issues arising from working in a non-smooth metric space.

We start with an estimate for the perimeter of the complement of the union of a finite number of balls.
\begin{lemma}\label{lem:ballsper}
	Let $\Xdm$ be a PI space. Suppose that $B_i\coloneqq B_{r_i}(x_i)\subset \X$, $i=1,...,k$, $k \in \nn,$ have all finite perimeter and satisfy $\Per(B_i,\essb B_j)=0$ for all $i\neq j$. Then
	\begin{equation}\label{eq:ballsper}
		\Per(B_1^c\cap B_2^c\cap...\cap B_k^c,.)\le \sum_{i=1}^k \Per(B_i,.) \restr{B_1^c\cap B_2^c\cap...\cap B_k^c}.
	\end{equation}
\end{lemma}
\begin{proof}
	We argue by induction on $k$. By \eqref{eq:ball perimeters} we have that for every ball $B\subset \X$ of finite perimeter
	\[
	\Per(B^c,.)=\Per(B,.)=\Per(B,.)\restr{B^c},
	\]
 which shows that the statement holds for $k=1$.  Suppose that the statement is true for some $k\in \nn$ and let $B_i\coloneqq B_{r_i}(x_i)\subset \X$, $i=1,...,k+1$ be as in the statement. By a repeated application of $(vi)$ in Lemma \ref{lem: essb and essi basic} we get
	\begin{equation}\label{eq:essb balls}
		\essb (B_1^c\cap B_2^c\cap...\cap B_k^c)\subset \essb B_i\cup ...\cup \essb B_k.
	\end{equation}
Using \eqref{eq:essb balls} and the assumption $\Per(B_{k+1},\essb B_j)=0$ for all $j\neq k+1$ gives
$$\Per(B_{k+1},\essb (B_1^c\cap B_2^c\cap...\cap B_k^c))=0.$$
Hence we can apply Proposition \ref{prop:per intersection} and obtain
\begin{equation}\label{eq:induction step}
\Per(B_1^c\cap B_2^c\cap...\cap B_{k+1}^c,.)\le \Per(B_{k+1},.) \restr{(B_1^c\cap B_2^c\cap...\cap B_k^c)^{(1)}}+\Per(B_1^c\cap B_2^c\cap...\cap B_k^c,.)\restr{(B_{k+1}^c)^{(1)}}.
\end{equation}
Since  balls are open sets, by using $(iii)$ and $(iv)$ in Lemma \ref{lem: essb and essi basic} it holds 
$$(B_1^c\cap B_2^c\cap...\cap B_k^c)^{(1)}\subset B_1^c\cap B_2^c\cap...\cap B_k^c,$$
that combined with \eqref{eq:induction step} and the induction hypothesis gives
\begin{align*}
	\Per(B_1^c\cap B_2^c\cap...\cap B_{k+1}^c,.)&\le \Per(B_{k+1},.) \restr{B_1^c\cap B_2^c\cap...\cap B_k^c}+\Per(B_1^c\cap B_2^c\cap...\cap B_k^c,.)\restr{(B_{k+1}^c)^{(1)}}\\
 &\le \Per(B_{k+1},.) \restr{B_1^c\cap B_2^c\cap...\cap B_k^c}+\Per(B_1^c\cap B_2^c\cap...\cap B_k^c,.)\restr{B_{k+1}^c}\\
	&\le  \Per(B_{k+1},.) \restr{B_1^c\cap B_2^c\cap...\cap B_k^c\cap B_{k+1}^c}+ \sum_{i=1}^k \Per(B_i,.)\restr{B_1^c\cap B_2^c\cap...\cap B_k^c\cap B_{k+1}^c},
\end{align*}
where in the second line we used that $(B_{k+1}^c)^{(1)}\subset B_{k+1}^c$ (recall $(iv)$ in Lemma \ref{lem: essb and essi basic}) and in the last line  that $\Per(B_{k+1},.)=\Per(B_{k+1},.)\restr{B_{k+1}^c}$ (recall \eqref{eq:ball perimeters}) for the first term and the induction hypothesis for the second term. This concludes the proof.
\end{proof}

Combining the above estimate with a covering argument we can prove the following proposition, from which Theorem \ref{thm:almost euclidean for small volumes PI} will easily follow. 
\begin{prop}[From local-to-global isoperimetric inequality]\label{prop:local to global}
		Let $\Xdm$ be a PI space. Suppose there exist constants $\lambda>0$, $\alpha\in(0,1]$ and a compact set ${\sf K}\subset \X$  such that  for all $x \in {\sf K}$ there exists $\rho(x)>0$ so that
		\begin{equation}\label{eq:local isop ass}
		    \Per (E)\ge \lambda \mm(E)^\alpha, \quad \forall \, E \subset B_{\rho(x)}(x)\, \text{ Borel.}
		\end{equation}
		Then there exists a constant $C=C({\sf K},\alpha,\lambda)$ such that 
  \begin{equation}\label{eq:patched isop}
       \Per (V)\ge  \lambda \mm(V\cap {\sf K})^\alpha-C\mm(V), \quad \forall \, V\subset \X \text{ Borel.}
  \end{equation}
\end{prop}
\begin{proof}
	We start by extracting once and for all a finite covering ${\sf K}\subset \cup_{i=1}^M B_{\frac{\rho(x_i)}2}(x_i)$, with $x_i \in {\sf K}$, and we set $\overline\rho \coloneqq \min_{i} \rho(x_i)>0$.  It is enough to prove \eqref{eq:patched isop} for sets $V$ of finite perimeter. Fix one such set $V.$
	
	 We claim that  there exist $r_1,\dots,r_M$, with  $r_i \in \left(\frac{\rho(x_i)}2,\rho(x_i)\right)$ such that the following hold: 
	\begin{enumerate}[label=\alph*)]
		\item\label{it:finite per} $B_{r_i}(x_i)$ has finite perimeter, for every $i=1,...,M,$
		\item\label{it:no inters} $\Per (V,\partial B_{r_i}(x_i))=0$, for every $i=1,...,M,$
  \item\label{it:good balls2} $\Per (B_{r_i}(x_i),V^{(1)})\le \frac{3\mea(V)}{\rho(x_i)}$, for every $i=1,...,M$,
		\item\label{it:good balls} $\Per (B_{r_i}(x_i),\partial B_{r_j}(x_j))=0$, for every $i,j=1,...,M$ with $i \neq j.$
	\end{enumerate}
It is sufficient to prove that:
\begin{enumerate}[label=(\roman*)]
		\item for any $i\in\{1,...,M\}$, there exists $A_i\subset \left(\frac{\rho(x_i)}2,\rho(x_i)\right)$ with  $\mathcal{H}^1(A_i)>0$ such that \ref{it:finite per},\ref{it:no inters} and \ref{it:good balls2} holds for every $r_i \in A_i$,
		\item for every $r_i>0$ such that $B_{r_i}(x_i)$ has finite perimeter, \ref{it:good balls} holds for every $j\neq i$  and for a.e.\ $r_j>0$.
	\end{enumerate}
 Indeed if these were true, up to removing from each $A_i$ a set of measure zero, we would have that  every choice $(r_1,\dots,r_M)\in A_i\times\dots,\times A_M$ satisfies all \ref{it:finite per},\ref{it:no inters}, \ref{it:good balls2} and \ref{it:good balls}. We start proving (i).
Fix  $i\in\{1,...,M\}$. From Proposition \ref{prop: coarea} we have that $B_r(x_i)$ has finite perimeter for a.e.\ $r>0$ (i.e.\ \ref{it:finite per}  holds for a.e.\ $r_i>0$). Moreover since $\Per(V,.)$ is a finite measure and $\{\partial B_{r}(x_i)\}_{r>0}$ are pairwise disjoint sets, we must have that $ \Per (V,\partial B_{r_i}(x_i))=0$ for a.e.\ $r_i>0$ (i.e.\ \ref{it:no inters}  holds for a.e.\ $r_i>0$). By applying Proposition \ref{prop: coarea} with $R\coloneqq \rho(x_i)$ we get
\[
\int_{\frac{\rho(x_i)}2}^{\rho(x_i)}\Per (B_{r}(x_i),V^{(1)})\, \d r\le \int_{0}^{\rho(x_i)}\Per (B_{r}(x_i),V^{(1)})\, \d r
\le  \mm(B_{\rho(x_i)}(x_i)\cap V^{(1)})=\mm(B_{\rho(x_i)}(x_i)\cap V),
\]
and by the Markov inequality  
$$\mathcal{H}^1\left (\left\{ r \in \left(\frac{\rho(x_i)}2,\rho(x_i)\right) : \  \Per (B_{r}(x_i),V^{(1)})> 3\mea(V\cap B_{\rho(x_i)}(x_i))/\rho(x_i) \right\}\right)\le \frac{\rho(x_i)}3,$$ which shows that the set 
$$\left\{ r \in \left(\frac{\rho(x_i)}2,\rho(x_i)\right) : \  \Per (B_{r}(x_i),V^{(1)})\le 3\mm(V\cap B_{\rho(x_i)}(x_i))/\rho(x_i) \right\}$$ 
has positive $\mathcal{H}^1$-measure, i.e.\ \ref{it:good balls2}  holds for every $r_i$ in a subset of $\left(\frac{\rho(x_i)}2,\rho(x_i)\right)$ of positive $\mathcal{H}^1$-measure. Combining all the above observations gives (i).
To show (ii) fix $i,j \in \{1,...,M\}$, $i \neq j$ and  $r_i>0$ such that  $B_{r_i}(x_i)$ has finite perimeter. As above  $\{\partial B_{r}(x_j)\}_{r>0}$ are pairwise disjoint sets, therefore we must have $\Per (B_{r_i}(x_i),\partial B_{r}(x_j))=0$ for a.e.\ $r>0.$ This proves (ii) and completes the proof of the claim.  

From now on we assume to have fixed $r_1,\dots r_M$ such that \ref{it:finite per},  \ref{it:no inters}, \ref{it:good balls2} and  \ref{it:good balls} above hold (note that this choice might depend on the set $V$) and we set $B_i\coloneqq B_{r_i}(x_i).$ By construction $K \subset \cup_{i=1}^M B_i.$  Consider the pairwise disjoints sets $\{U_i\}_{i=1}^M$ defined inductively as follows:

\[
U_1\coloneqq B_1, \quad  U_i\coloneqq B_i\cap (B_{i-1}^c\cap...\cap B_1^c), \quad \forall \, i=2,...,M.
\]
Clearly $\{U_i\}_{i=1}^M$ is a family of disjoint Borel sets which is a covering of $K$. We claim that 
\begin{equation}\label{eq:good covering}
	\Per(V,\essb U_i)=0, \quad \forall \, i=1,...,M.
\end{equation}
Indeed, from $(vi)$ and $(ii)$ of Lemma \ref{lem: essb and essi basic} one infers that 
\begin{equation*}
    \essb U_i \subset \essb B_1\cup \essb (B_2^{c})\cup...\cup \essb (B_{i-1}^c)\subset  \partial B_1\cup...\cup \partial B_{i-1}.
\end{equation*}
 From this \eqref{eq:good covering} follows recalling \ref{it:no inters}. Thanks to \eqref{eq:good covering} we are in position to apply Proposition \ref{prop:per intersection} to deduce that
\begin{equation}\label{per:v int ui}
	\Per(V\cap U_i)\le \Per (V,U_i^{(1)})+\Per (U_i,V^{(1)}), \quad \forall \, i=1,...,M.
\end{equation}
The goal is now to give an upper bound on each term on the right-hand side of \eqref{per:v int ui}.
Since $U_i$ are pairwise disjoint by construction, by point $(v)$ of Lemma \ref{lem: essb and essi basic} it follows that also the sets $U_i^{(1)}$ are pairwise disjoint, hence
\begin{equation}\label{eq:Ui interior disjoint}
	\sum_{i=1}^M\Per (V,U_i^{(1)})\le \Per (V).
\end{equation}
To estimate $	\Per (U_i,V^{(1)})$ we note that from \ref{it:good balls} it holds  
$$\Per(B_i,\essb (B_1^c\cap...\cap B_{i-1}^c))=0,$$ indeed  $\essb (B_1^c\cap...\cap B_{i-1}^c)\subset \partial B_1\cup...\cup \partial B_{i-1}$ (recall $(ii)$ in Lemma \ref{lem: essb and essi basic}). Hence, recalling that by construction $U_i=B_i\cap (B_{i-1}^c\cap...\cap B_1^c)$, we can apply again Proposition \ref{prop:per intersection} to get
\[
\Per(U_i,V^{(1)})\le 	\Per (B_i,V^{(1)})+ \Per(B_1^c\cap...\cap B_{i-1}^c,V^{(1)})\restr{B_i^{(1)}}.
\]
From this and Lemma \ref{lem:ballsper}
\begin{align*}
		 \Per(U_i,V^{(1)}) &\overset{\eqref{eq:ballsper}}{\le}  \Per (B_i,V^{(1)})+  \sum_{j=1}^{i-1} \Per (B_j,V^{(1)})\restr{B_i^{(1)}\cap B_1^c\cap...\cap B_{i-1}^c}\\
			&\le   \Per (B_i,V^{(1)})+  \sum_{j=1}^{i-1} \Per (B_j,V^{(1)})\restr{B_i\cap B_1^c\cap...\cap B_{i-1}^c}\\
			&=  \Per (B_i,V^{(1)})+  \sum_{j=1}^{i-1} \Per (B_j,V^{(1)})\restr{U_i},
\end{align*}
where in the second line we used that $B_i^{(1)}\subset B_i\cup \partial B_i$ and that $\Per(B_j,\partial B_i)=0$ for all  $j\neq i$, thanks to \ref{it:good balls}.
Summing in $i$ and recalling that $U_i$ are disjoint 
\begin{equation}\label{eq:per ui sum}
	\begin{split}
		 \sum_{i=1}^M\Per(U_i,V^{(1)})&\le \sum_{i=1}^M\Per (B_i,V^{(1)})+   \sum_{i=1}^M\sum_{j=1}^{i-1} \Per (B_j,V^{(1)})\restr{U_i}\\
		 &\le  \sum_{i=1}^M\Per (B_i,V^{(1)})+   \sum_{j=1}^{M} \sum_{i=1}^M \Per (B_j,V^{(1)})\restr{U_i}\\
		 &\le 2 \sum_{i=1}^M\Per (B_i,V^{(1)})\overset{\ref{it:good balls2}}{\le}  6\sum_{i=1}^M \frac{\mm(V)}{\rho(x_i)}\le 6M \frac{\mm(V)}{\bar \rho}.	\end{split}
\end{equation}
Combining \eqref{eq:per ui sum}, \eqref{eq:Ui interior disjoint} and \eqref{per:v int ui} we get
\begin{equation}\label{eq:upper bound per}
		\sum_{i=1}^M \Per(V\cap U_i)\le   \Per (V)+ 6M \frac{\mm(V)}{\bar \rho}.
\end{equation}
On the other hand $V\cap U_i\subset B_i\subset B_{\rho(x_i)}(x_i)$, hence by assumption \eqref{eq:local isop ass}
\[
\sum_{i=1}^M \Per(V\cap U_i)\ge \lambda \sum_{i=1}^M \mm(V \cap U_i)^{\alpha}\ge \lambda  \left(\sum_{i=1}^M \mm(V \cap U_i)\right)^{\alpha}\ge \lambda \mm(V\cap {\sf K})^\alpha,
\]
since the function $x\mapsto x^{\alpha}$ is subadditive and the sets $\{U_i\}_{i=1}^M$ cover ${\sf K}$. This combined with \eqref{eq:upper bound per} yields
\[
 \Per (V)\ge  \lambda \mm(V\cap {\sf K})^\alpha-6M \frac{\mm(V)}{\bar \rho}.
\]
The constants $M$ and $\bar \rho$ depend only on the initial choice of the covering $B_{\frac{\rho(x_i)}2}(x_i)$ and thus depend only on $K,$ $\alpha$ and $\lambda,$ (and not on $V$). This concludes the proof.
\end{proof}

An application of  Proposition \ref{prop:local to global} yields immediately the main result of this section.
\begin{proof}[Proof of Theorem \ref{thm:almost euclidean for small volumes PI}]
Fix $\eps\in(0,1)$ and $\eta>0.$ Let $C_{\eps/2}\subset \X$ be as in the statement.
Since $C_{\eps/2}$ is closed, by upper regularity there exists an open set $U_{\eps,\eta}$ containing $C_{\eps/2}$ and such that $\mm(U_{\eps,\eta})<\eta.$ 
Set ${\sf K}_{\eps,\eta}\coloneqq \X\setminus U_{\eps,\eta}$, which is compact (because it is closed and bounded and $\X$ is a proper, being a PI space). In particular \eqref{eq:local eucl pi} holds also for every $x \in K_{\eps,\eta}$. Therefore the hypotheses of Proposition \ref{prop:local to global} are satisfied with $K=K_{\eps,\eta}$, $\alpha=\frac{N-1}{N}$ and $\lambda=(1-\eps)N\omega_N^{\frac1N}$ and we deduce that there exists a constant $C$ depending only on $K_{\eps,\eta}, N$ and $\eps$ (and thus only  on $\X,$ $\eps$, $N$ and $\eta$) such that 
\begin{align*}
		\Per (E)&\ge  (1-\eps/2)N\omega_N^{\frac1N}\mm(E\cap {\sf K_{\eps,\eta}})^{\frac{N-1}{N}}-C\mm(E)\\
		&=  (1-\eps/2)N\omega_N^{\frac1N}\mm(E)^{\frac{N-1}{N}} \left[\left(1-\frac{\mea(E\cap U_{\eps,\eta})}{\mea(E)}\right)^{\frac{N-1}{N}}-\hat{C}\mm(E)^{\frac1N}\right]\\
  &\ge (1-\eps/2)N\omega_N^{\frac1N}\mm(E)^{\frac{N-1}{N}} \left[1-\frac{\mea(E\cap U_{\eps,\eta})}{\mea(E)}-\hat{C}\mm(E)^{\frac1N}\right] , \quad \forall E\subset \X \text{ Borel,}
	\end{align*}
where $\hat C=C((1-\eps/2)N\omega_N^{\frac1N})^{-1}$.
From this we obtain that the conclusion of theorem holds  taking $	 \beta\coloneqq \hat C^{-N}\delta_\eps^N$ and $\beta'\coloneqq \delta_\eps,$ where $\delta_\eps\coloneqq \frac12 \frac{\eps}{2-\eps}.$
\end{proof}

\section{Almost Euclidean Faber-Krahn inequality for small volumes}
Similarly to the classical Faber-Krahn inequality in $\rr^N$, combining   the almost Euclidean isoperimetric inequality for small volumes in Section \ref{sec:isop} and the P\'olya-Szeg\H{o} inequality of Section \ref{sec:ps}, we deduce here  an  almost-Euclidean Faber-Krahn inequality for small volumes similar to \cite[Lemme 16]{BM82} in the case of Riemannian manifolds. However it is not possible to apply directly the P\'olya-Szeg\H{o} inequality as in \cite{BM82}, because our isoperimetric inequality applies only to sets that have small volumes \emph{and avoid} a bad set with small measure.
This technical difficulty will require a more careful argument, which will eventually lead to a Faber-Krahn inequality that  applies only to sets that again avoid a portion of the space with small measure (see Theorem \ref{thm:faber}).

We start by recalling the well known expression of the first Dirichlet eigenvalue of a ball  in the $N$-dimensional Euclidean space (see e.g. \cite{BM82}):
\begin{equation}\label{eq:lambda eucl ball}
\lambda_1(B_r^{\rr^N}(0))=\left(\frac{\omega_N}{\Leb^N(B_r^{\rr^N}(0))}\right)^{2/N}j^2_{\frac{(N-2)}{N}}, \qquad\forall \, r>0,
\end{equation}
where $j_{\frac{(N-2)}{N}}$ denotes the first positive zero of the Bessel function (of the first kind) of index $\frac{(N-2)}{N}$. 

Next we obtain a weaker Faber-Krahn inequality, i.e.\ with a rough constant, but that applies to sets also with large volume.
\begin{prop}[Faber-Krahn inequality in PI spaces]\label{prop:Faber Kran PI}
	Let $\Xdm$ be a bounded PI space satisfying for some $N\in \nn$ and some constant $c>0$
	\[
	\frac{\mea(B_r(x))}{\mea(B_R(x))}\ge c\left(\frac{r}{R}\right)^N, \quad \forall\, x \in \X,\, \forall\, 0<r<R.
	\]
	Then there exist constants $v_0=v_0(\X)>0$ and $C=C(\X,N)>0$ such that
	\begin{equation}\label{eq:general FK}
	\lambda_1(\Omega)\ge \frac{C}{\mea(\Omega)^\frac 2N}, \quad \forall\, \Omega\subset \X \text{ open, } 0<\mea(\Omega)\le v_0.
	\end{equation}
\end{prop}
\begin{proof}
	By Proposition \ref{prop: class isop} there exist  constants $w_0=w_0(\X)>0$ and $C_I=C_I(\X,N)>0$ such that
	\begin{equation}\label{eq:isop in proof fk}
	\Per(E) \ge C_I \mm(E)^{\frac{N-1}N}, \qquad \forall\,   E \subset \X \text{ Borel  such that $\mea(E)\le w_0$}
	\end{equation}
	In particular if $\mea(\Omega)\le v_0\coloneqq w_0$, then \eqref{eq:isop in proof fk} holds for every $E\subset \Omega$ Borel. Moreover taking  $v_0<\mea(\X)$ we have $\Omega \neq \X.$
	Therefore we can apply the P\'olya-Szeg\H{o} inequality \eqref{eq:euclpolya} and deduce 
	\begin{equation}
	\frac{\int_{\Omega}|{D u}|^2\d \mm}{\int_{\Omega}|u|^2\d \mm}\ge \Big(\frac{C_I}{N\omega_N^{1/N}}\Big)^2\frac{\int_{\Omega^*}|{D u_{N}^*}|^2\d \Leb^N}{\int_{\Omega^*}|u_N^*|^2\d \Leb^N}\ge  \frac{C}{\mea(\Omega)^{\frac{2}{N}}}, \quad \forall u \in \W_0(\Omega),\,u\ge 0,\, u\not \equiv 0,
	\end{equation}
	where we have also used that $\|u\|_{L^2(\Omega,\mm)}=\|u_N^*\|_{L^2(\Omega^*,\Leb^N)}$, $\mm(\Omega)=\Leb^N(\Omega^*)$, and in the last passage the identity \eqref{eq:lambda eucl ball} and the fact that $u_N^*\in \W_0(\Omega^*)$. By taking the infimum with respect to all the possible $u$ in the characterization \eqref{char: lambda1}, we get the result.
\end{proof}

From the previous proposition we can deduce the following version of the Faber-Krahn in the Euclidean setting. Even if it will be not  used in this note, we think it is worth to be isolated in a separate statement.  Indeed it has been pointed out repeatedly in the previous literature that one of the major difficulties in counting nodal domains for non-Dirichlet boundary condition  in subset of $\rr^N$ with irregular boundary is the absence of a  suitable Faber-Krahn inequality for subdomains close to the boundary (see e.g.\ discussions in \cite[Section 2]{HaSh23}, \cite[Section 1.2]{Le19}, \cite[Section 1.2]{GiLe20}, \cite[Section 1.1]{BCM23}). This was one of the main issues faced \cite{Le19} which also forced the assumption of a $C^{1,1}$ boundary  (see also \cite{HaSh23}).
Here we show precisely that a Faber-Krahn-type inequality does hold in any uniform domain, no matter how close is the support of the function to the boundary.
\begin{cor}[Faber-Krahn inequality for uniform domains]\label{cor:FKeucl}
    Let $\Omega \subset \rr^N$ be a uniform domain. Then there exist  constants $v_0 \in (0,\Leb^N(\Omega))$ and $C>0$, depending only on $\Omega,$ such that 
    \[
    \frac{\int_{\Omega} |\nabla u|^2\d \Leb^N }{\int_{\Omega} u^2\d \Leb^N}\ge \frac{C}{(\Leb^N(\supp(u)))^{\frac 2N}}, \quad \forall u \in \W(\Omega), u\not\equiv 0, \ \text{such that } \Leb^N(\supp(u))\le v_0. 
    \]
\end{cor}
\begin{proof}
    Let $u\in \W(\Omega)$. Then by Theorem \ref{thm:sobolev compatibility} there exists $\tilde u \in L^2(\overline \Omega,\Leb^N)$ such that $\tilde u =u$ $\Leb^N$-a.e.\ in $\Omega$ and such that $\tilde u\in \W(\overline \Omega,\sfdomega,\Leb^N\restr{\overline \Omega})$  where $\sfd$ denotes the Euclidean distance (recall also Remark \ref{rmk:eucl lapl}). Moreover by Theorem \ref{th:almost-iso} the m.m.s.\ $\Xdm\coloneqq (\overline \Omega,\sfdomega,\Leb^N\restr{\overline \Omega})$  satisfies the hypotheses of Proposition \ref{prop:Faber Kran PI}.  Finally by Lemma \ref{lem:zero boundary} we have $\Leb^N(\partial \Omega)=0.$  Let $v_0<\mea(\X)=\Leb^N(\Omega)$ and $C>0$ be the constants given by Proposition \ref{prop:Faber Kran PI}, which depend only on $\Omega.$ 
     Consider the open set $U_\eps\coloneqq (\supp(\tilde u))^\eps$, $\eps>0$. Then $\Leb^N(U_{\eps})\to \Leb^N(\supp (\tilde u))=\Leb^N(\supp( u))$ as $\eps \to 0,$ having also used that $\Leb^N(\partial \Omega)=0.$ Hence assuming that  $\Leb^N(\supp( u))<v_0(\X)$ we have that $\mm(U_{\eps})\le v_0$ for $\eps$ small enough. Clearly $\tilde u \in \W_0(U_{\eps})$ (recall \eqref{eq:support inside}), hence we can apply \eqref{eq:general FK} and obtain
     \[
     \frac{\int_{\overline \Omega} |\nabla \tilde u|^2\d \Leb^N }{\int_{\overline \Omega} \tilde u^2\d \Leb^N}\ge \frac{C}{\mea(U_\eps)^\frac 2n}.
     \]
     Letting $\eps\to 0$ and recalling that $\int_{\overline \Omega} |\nabla \tilde u|^2\d \Leb^N=\int_\Omega |\nabla u|^2\d \Leb^N$ concludes the proof. 
\end{proof}

We pass to the statement of our main  Faber-Krahn inequality for small volumes.
\begin{theorem}[Almost Euclidean Faber-Krahn inequality for small volumes]\label{thm:faber}
	Let $\Xdm$ be a bounded PI space satisfying for some $N\in \nn$ and some constant $c>0$
	\[
	\frac{\mea(B_r(x))}{\mea(B_R(x))}\ge c\left(\frac{r}{R}\right)^N, \quad \forall\, x \in \X,\, \forall 0<r<R.
	\]
	Suppose that for every $\eps>0$ there exists a closed set $C_\eps\subset \X$ with $\mea(C_\eps)=0$ such that for every $x \in \X\setminus C_\eps$ there exists a constant $\rho=\rho(x,N,\eps)>0$  satisfying
	\begin{equation}\label{eq:local eucl pi in fk}
	\Per(E)\ge(1-\eps)N\omega_N^{\frac1N}\mea(E)^{\frac{N-1}{N}}, \quad \forall \, E \subset B_{\rho}(x)\, \text{ Borel,}
	\end{equation} Then for every $\eps\in(0,1)$ and $\eta>0$ there exists an open set $U_{\eps,\eta}$ with $\mm(U_{\eps,\eta})<\eta$ and constants $\delta=\delta(   \X, \eps,N,\eta)>0$, $\delta'=\delta'(\X,\eps,N)>0$ such that
	for every $\Omega\subset \X$ open satisfying
	\[
	\mea(\Omega)\le \delta, \quad \frac{\mea(\Omega\cap U_{\eps,\eta})}{\mea(\Omega)}\le \delta',
	\]
and, denoted by $\Omega^*\coloneqq B_r(0)\subset \rr^N$ the ball satisfying $\mm(\Omega)=\Leb^N(\Omega^*)$, it holds
 \begin{equation}\label{eq:faber}
	\lambda_1(\Omega)\ge (1-\eps){\lambda_1}(\Omega^*).
	\end{equation}
\end{theorem}
The key point of Theorem \ref{thm:faber} is that the constant $\delta'$ \emph{does not depend} on $\eta$. This will be crucial in the proof of the main result, which will be done in Section \ref{sec:final}. Indeed we will eventually need to get rid of $\delta'$  by  sending $\eta \to 0 $ (see in particular \eqref{eq:fine}).

\begin{remark}\label{rmk:assumptions}
    Thanks to  Theorem \ref{th:almost-iso} we know that, given  $(\Y,\tilde \sfd,\hau^N)$ an $\RCD(K,N)$ space and  $\Omega\subset \Y$  a uniform domain,  the metric measure space $\Xdm\coloneqq (\overline\Omega,\tilde \sfd\restr{\overline \Omega}, \hau^N\restr {\overline\Omega})$ satisfies the hypotheses of the above Theorem \ref{thm:faber}. This will actually be the way in which we will apply this result in the sequel.
    \fr
\end{remark}

In the proof of Theorem \ref{thm:faber} we will make use of the following elementary observation.
\begin{lemma}\label{lem:lower bound rayleigh}
		Let $\Xdm$ be a m.m.s. and  $f\in\LIP_{c}(X)$, $f\not\equiv 0 $, $f\geq 0$. It holds
		\begin{equation*}
		\frac{\int_{X}|Df|^2\,\d\,\mm}{\int_{X}|f|^2\,\d\mm}\geq \lambda_1(\{f>0\}),
		\end{equation*}
		where $\{f>0\}\coloneqq\{x\in \X\colon f(x)>0\}$.
	\end{lemma}
	\begin{proof}
		Let $\psi_n\coloneqq(f-\frac{1}{n})^{+}$. Then $\psi_n\in \LIP_{c}(\{f>0\})$  and for $n\in \mathbb{N}$ sufficiently large, $\psi_n\not\equiv 0$. Then for $n$ large $\psi_n$ is a competitor in \eqref{def: lambda1}, so 
		\begin{equation*}
		\lambda_1(\{f>0\})\leq \liminf_{n\to +\infty}\frac{\int_{X}|D\psi_n|^2\,\d\mm}{\int_{X}|\psi_n|^2\,\d\mm}=
  \liminf_{n\to +\infty}\frac{\int_{\{f\ge 1/n\}}|Df|^2\,\d\mm}{\int_{\{f\ge 1/n\}}|f-1/n|^2\,\d\mm}=\frac{\int_{X}|Df|^2\,\d\mm}{\int_{X}|f|^2\,\d\mm},
		\end{equation*}
		where the first equality follows   from the locality of the weak upper gradient. 
	\end{proof}
 	
We are now ready to prove the Faber-Krahn inequality for small volumes.
\begin{proof}[Proof of Theorem \ref{thm:faber}]
	Fix $\eps\in(0,1)$ and $\eta>0$. Let $ \beta=\beta(\X,\frac{\eps}{2},N,\eta)>0, \beta'=\beta'(\frac{\eps}{2})>0, U_{\frac{\eps}{2},\eta}\subset \X$ be respectively the constants and the set as given in Theorem \ref{thm:almost euclidean for small volumes PI} and recall that $\mea(U_{\frac{\eps}{2},\eta})<\eta.$ In the following we will simply write $\beta$ and $\beta'$ to denote these constants and write $U$ to denote the set $U_{\frac{\eps}{2},\eta}$.
	
	Let  $\delta,\delta'\in(0,1)$ be constants small enough to be chosen later and in such a way that  $\delta$ will depend in the end only on $ \X, \eps,N,\eta$, while $\delta'$ only on $\X,N,\eps.$  
 
	Fix $\Omega\subset \X$ open such that
	$$\mea(\Omega)\le \delta, \quad \frac{\mea(\Omega\cap U)}{\mea(\Omega)}\le \delta',$$
 and define $\Omega^*\coloneqq B_r(0)\subset \rr^N$ where $r>0$ is so that $\Leb^N(\Omega^*)=\mea(\Omega).$ Up to choosing $\delta<\mea(\X)$ we can also assume that $\Omega \neq \X.$
 
	Let $u\in \LIP_c(\Omega)$ be a competitor in the infimum of \eqref{eq:non-van lambda1}. In particular $\ u\not\equiv 0$,  $u\ge 0$ and $\lip(u)\neq0$ $\mm$-a.e.\  in $\{u>0\}$.
 
	We divide two cases:
	
\noindent \textsc{Case 1}: $\mea(\{u>0\})\le3\sqrt{\delta'}\mea(\Omega)$. From Lemma \ref{lem:lower bound rayleigh} and Proposition \ref{prop:Faber Kran PI} we have 
	\begin{equation}\label{eq:case1 proof}
	\frac{\int_\X |D u|^2 \d\mm}{\int_\X u^2\d\mm}\geq \lambda_1(\{u>0\}) \geq \frac{C(\X,N) }{ \mea({u>0})^{\frac{2}{N}}}	\geq\frac{C(\X,N) (3\sqrt{\delta'} )^{-\frac{2}{N}}}{ \mea(\Omega)^{\frac{2}{N}}}\ge \frac{2 (j_{\frac{N-2}{N}})^2 \omega_N^{\frac{2}{N}}}{\Leb^N(\Omega^*)^{\frac{2}{N}}} \overset{\eqref{eq:lambda eucl ball}}{=} 2\lambda_1(\Omega^*),
	\end{equation}
	with $C(\X,N)>0$ is the constant given by Proposition \ref{prop:Faber Kran PI}, which can be applied if $3\sqrt{\delta'}\delta \le   v_0(\X)$  (where $v_0(\X)$ is given by Proposition \ref{prop:Faber Kran PI}) and the last inequality in \eqref{eq:case1 proof} holds provided $3\sqrt{\delta'} <\frac{C(\X,N)^{N/2}}{2^{N/2} (j_{\frac{N-2}{N}})^N \omega_N}$. 
	
\noindent \textsc{Case 2}: $\mea(\{u>0\})>3\sqrt{\delta'} \mea(\Omega)$. Set  
	\begin{align*}
	s\coloneqq \sup \{ t\,\colon\, \mea(\{u>t\})\ge2\sqrt{\delta'} {\mea(\Omega)}\}
	\end{align*}
	and observe that $s>0$ and that
 \begin{equation*}
 \begin{split}
       &\mea(\{u>s\})\le2\sqrt{\delta'}\mea(\Omega),\\
     &\mea(\{u>t\})\ge2\sqrt{\delta'}\mea(\Omega), \quad \forall \, t<s.
 \end{split}
 \end{equation*}
The first one follows because $\mea(\{u>s\})=\lim _{t\to s^+} \mea(\{u>t\})\le2\sqrt{\delta'}\mea(\Omega)$, while for the second note that $t\mapsto \mea(\{u>t\})$ is monotone non-increasing.
	Set $\hat u\coloneqq u \wedge s$ and $\tilde u\coloneqq (u-s)^+$, so that $u=\hat u +\tilde u$ and $\hat u$, $\tilde u$ are in $\LIP_c(\Omega)$ and $\hat u\geq 0$, $\tilde u\geq 0$, $\hat u\not\equiv 0$, $\tilde u\not	\equiv 0$. Observe that if multiply $u$ by a constant $c>0$, $u$ still satisfies the hypotheses of \textsc{Case 2} and the number $s$ defined above gets also multiplied by $c.$ Hence also $\hat u$ gets multiplied by $c.$ Therefore, since the value $\frac{\int |D u|^2 \d\mm}{\int u^2\d\mm}$ is scaling invariant, up to multiplying $u$ by a constant we can assume that $\int_\X \hat{u}^2\d \mm=1.$ Then
	\begin{align}\label{eq:split rayleigh}
	\frac{\int_\X |Du|^2 \d\mm}{\int_\X u^2 \d\mm}=\frac{\int_\X |D \hat u|^2+|D\tilde u|^2 \d\mm}{\int_\X \hat{u}^2 \d\mm+\int_\X \tilde u^2+2 \int_\X \hat u\tilde u \d\mm}\ge \frac{\int_\X |D \hat u|^2+|D \tilde u|^2 \d\mm}{1+\int_\X \tilde u^2 \d\mm+2\sqrt{ \int_\X\tilde u^2\d\mm} }.
	\end{align}
	Using again Lemma \ref{lem:lower bound rayleigh}  applied with $f=\tilde u$  and noting that $\{\tilde u>0\}=\{u>s\}$  we have 
	\begin{equation}\label{eq:rayleigh tilde u}
	\frac{\int_\X |D \tilde u|^2 \d\mm}{\int_\X | \tilde u|^2\d\mm}\geq \lambda_1(\{u>s\}) \ge \frac{C(\X,N)}{\mea(\{u>s\})^\frac2N} \ge \frac{C(\X,N) }{ (2\sqrt{\delta'} \mea(\Omega))^{\frac{2}{N}}},
	\end{equation}
where the second inequality follows from  Proposition \ref{prop:Faber Kran PI} (which as above can be applied  provided $2\sqrt{\delta'}\delta \le v_0(\X)$).
	Moreover, since $\mea(\{u> t\})\ge 2\sqrt{\delta'}\mea(\Omega)$, for every $t < s$, we have
	\begin{equation*}
	\frac{\mea(\{u>t\}\cap U)}{\mea(\{u>t\})}\le \frac{\mea(\Omega \cap U)}{\mea(\{u>t\})}\le \frac{{\delta'}\mea(\Omega)}{2\sqrt{\delta'}\mea(\Omega)}=\frac{{\sqrt{\delta'}}}{2}<\beta', \quad\forall\, t<s,
	\end{equation*}
	provided $\delta'\le (\beta')^2$ (recall that $\beta'$ depends only on $\eps$). We also have $\mea(\{u>t\})\le \mea(\Omega)\le \delta\le \beta$,
 provided $\delta \le \beta.$ 
	Therefore we can apply Theorem \ref{thm:almost euclidean for small volumes PI}  to the set $E=\{u>t\}$ and obtain
	\[
	\Per(\{u> t\})\ge (1-\frac{\eps}{2})N\omega_N^{\frac1N}\mea(\{u>t\})^{\frac{N-1}{N}},\quad \forall\, t< s.
	\]
	We can then apply the  P\'olya-Szeg\H{o} inequality in  point $i)$ of  Theorem \ref{thm:eu_polyaszego} to the function $\hat u$ (note that assumption \eqref{eq:isop_eu} is satisfied by provided $\delta\le w_0(X)$, where $w_0(\X)$ is the constant given by Proposition \ref{prop: class isop}), to get 
 \begin{align*}
     \int |D \hat u|^2 \d \mm &= \int_{\{u\le s\}}|D u|^2\d \mm \ge(1-\frac \eps2)^2 \int_0^s \int_{\rr^N} |D u^*_N|\d\Per(\{u^*_N>t\})\d t \\
     &=(1-\frac \eps2)^2\int_{\{u^*_N\le\, s \}} |D u^*_N|^2\d\Leb^N=(1-\frac \eps2)^2\int_{\rr^N} |D (u^*_N\wedge s)|^2 \d \Leb^N,
 \end{align*}
 where $u^*_N\in \LIP_c(\Omega^*)$ is the Euclidean monotone rearrangement of $ u$ (see Definition \ref{def:eucl rearrangement}) and where in the second to last equality we used the coarea formula in the Euclidean space (see e.g.\ \cite[Theorem 18.1]{Maggi}).
	Moreover since $ u^*_N\wedge s \in \LIP_c(\Omega^*)$  we have 
	$$\int_{\rr^N}  |D (u^*_N\wedge s)|^2 \d \Leb^N\ge \lambda_1(\Omega^*)\int_{\rr^N}  |u^*_N\wedge s|^2\d \Leb^N=\lambda_1(\Omega^*)\int_\X |\hat u|^2\d \mm=\lambda_1(\Omega^*),$$
 where in the first identity we used the equimeasurability of $u$ and $u_N.$
	Hence
	\begin{equation}\label{eq:rayleigh hat u}
	\frac{\int |D \hat u|^2 \d\mm}{\int_\X \hat u^2\d\mm}\ge(1-\frac \eps2)^2 \lambda_1(\Omega^*).
	\end{equation}
	Therefore we can plug \eqref{eq:rayleigh tilde u} and \eqref{eq:rayleigh hat u} into \eqref{eq:split rayleigh} to get 
	\begin{equation*}
	\frac{\int_\X |D u|^2\d \mm}{\int_\X u^2\d \mm}\ge \frac{(1-\frac{\eps}{2})^2\lambda_1(\Omega^*)+\frac{C(\X,N)}{(2\sqrt{\delta'}\mea(\Omega))^{\frac2N}}A^2}{1+A^2+2A},
	\end{equation*}
	where $A\coloneqq \left(\int_\X \tilde u^2\d\mm\right)^\frac12.$ We now minimize in $A\in [0,\infty).$ To do so we observe that the function $f(t)\coloneqq \frac{a+bt^2}{1+t^2+2t}$, $a,b> 0,$ has derivative $f'(t)=\frac{2(bt-a)}{(1+t)^3}$. Hence $f$ has a global minimum in $[0,\infty)$ at $t=\frac ab$ of value $f(\frac ab)=\frac{a}{1+\frac ab}.$
	Therefore we obtain
	\begin{align*}
	\frac{\int_\X |D u|^2\d \mm}{\int_\X u^2\d \mm}&\ge \frac{(1-\frac{\eps}{2})^2\lambda_1(\Omega^*)}{1+(1-\frac{\eps}{2})^2\lambda_1(\Omega^*)(2\sqrt{\delta'}\mea(\Omega))^{\frac2N}C(\X,N)^{-1}}\\
    &\hspace{-2mm}\overset{\eqref{eq:lambda eucl ball}}{=}
	\frac{(1-\frac{\eps}{2})^2\lambda_1(\Omega^*)}{1+(1-\frac{\eps}{2})^2 (j_{\frac{N-2}{N}})^2 (\omega_N2\sqrt{\delta'})^{\frac2N}C(\X,N)^{-1}}.
	\end{align*}
	Choosing $\delta'$ small enough, depending only on $\eps,N$ and $\X$,  such that 
	\begin{equation*}
	\frac{(1-\frac{\eps}{2})^2}{1+(1-\frac{\eps}{2})^2 (j_{\frac{N-2}{N}})^2 (\omega_N2\sqrt{\delta'})^{\frac2N}C(\X,N)^{-1}}>1-\eps,
	\end{equation*}
	we get the conclusion. 
\end{proof}

\section{Proof of main theorem}\label{sec:final}
In this part we prove Theorem \ref{main th} and Corollaries \ref{cor: mainRCD}, \ref{cor: mainEucl} combining the results of all the previous sections. One last ingredient, contained in the next statement, is a crucial inequality relating the eigenvalue of an eigenfunction with the first Dirichlet eigenvalue of one of its nodal domains. This can be seen as a generalization of  Lemme 2 in Appendix D of \cite{BM82} proved there in the setting of Riemannian manifolds. Recall also that eigenfunctions of the Laplacian in PI spaces  are continuous (see Theorem \ref{thm:continuous eigen}).
\begin{prop}\label{prop: neum bound dir}
		Let $\Xdm$ be a bounded infinitesimally Hilbertian PI space, $U \subset \X$ be open  and   $f$ be a Dirichlet or Neumann eigenfunction of the  Laplacian in $U$ of eigenvalue $\lambda$. If $\Omega\subset U$ is a nodal domain of (the continuous representative of) $f$,  then $\Omega$ is open in $\X$ and it holds
  \begin{equation}\label{eq:NDineq}
      \lambda_1(\Omega)\leq\lambda=\frac{\int_{\Omega} |D f|^2\, \d \mm}{\int_{\Omega} f^2\,\d\mm}.
  \end{equation}
	\end{prop}
 In \cite{BM82} (in Riemannian setting) it is  shown that the first in \eqref{eq:NDineq} is  actually an equality in the Dirichlet case, however we do not know whether the same is true also in this more general setting. Nevertheless \eqref{eq:NDineq} will be sufficient for our purposes.

Specializing  Proposition \ref{prop: neum bound dir} to the Euclidean setting we also obtain the following result which, even if not needed in the sequel, we believe it is interesting on its own. In particular it extends previous results in \cite[Proposition 1.7]{Le19} and \cite[Lemma 3.3]{BCM23}, where the same was proved respectively for $C^{1,1}$ domains and  for planar piecewise smooth domains.
\begin{cor}[Green's formula for eigenfunctions]\label{cor:rayleig-on-nodal}
    Let $\Omega\subset \rr^N$ be a uniform domain and let $f$ be a Neumann eigenfunction in $\Omega$ of eigenvalue $\lambda$. Then for every $U$ nodal domain of $f$ it holds
    \[
    \int_{U} |\nabla f|^2\, \d \Leb^N=\lambda \int_U f^2 \d \Leb^N.
    \]
\end{cor}
\begin{proof}
Let $f$ be as in the statement and $U$ be a nodal domain of $f.$
    From Corollary \ref{cor:spectrum compatilbity} there exists $\tilde f \in L^2(\overline \Omega,\Leb^N)$ such that $\tilde f =f$ $\Leb^N$-a.e.\ in $\Omega$ and such that $\tilde f$ is an eigenfunction for the Laplacian of eigenvalue $\lambda$ in the m.m.s.\ $(\overline \Omega,\sfdomega,\Leb^N\restr{\overline \Omega})$  where $\sfd$ denotes the Euclidean distance (recall also Remark \ref{rmk:eucl lapl}). Moreover $(\overline \Omega,\sfdomega,\Leb^N\restr{\overline \Omega})$ is  an infinitesimally Hilbertian PI space thanks to Theorem \ref{thm:uniform implies PI}. Hence $\tilde f$ is continuous in $\overline \Omega$ (recall Theorem \ref{thm:continuous eigen}).   Proposition \ref{prop:components compatibility} then says that the set $\phi(U)\coloneqq U \cup (\partial U\cap \partial \Omega \setminus \{f=0\})$ is a nodal domain of $\tilde f$. Finally by Lemma \ref{lem:zero boundary} we have $\Leb^N(\partial \Omega)=0.$ Therefore applying Proposition \ref{prop: neum bound dir} to $\Xdm=(\overline \Omega,\sfdomega,\Leb^N\restr{\overline \Omega})$ and $\tilde f$ we get
    \[
     \lambda\int_{ U}  f^2\,\d\Leb^N=\lambda\int_{\phi(U)} (\tilde f)^2\,\d\Leb^N\overset{\eqref{eq:NDineq}}{=}\int_{\phi(U)} |D \tilde f|^2\, \d \Leb^N=\int_{U} |D \tilde f|^2\, \d \Leb^N=\int_{U} |\nabla f|^2\, \d \Leb^N,
    \]
    where the last equality follows by Theorem \ref{thm:sobolev compatibility}. This concludes the proof.
\end{proof}

	\begin{proof}[Proof of Proposition \ref{prop: neum bound dir}]
		Let $f$ be a Dirichlet or Neumann eigenfunction of the Laplacian in $U$ of eigenvalue $\lambda$ and let $\Omega$ be one of its nodal domains. Since $\Xdm$ is a PI space, the metric space $(\X,\sfd)$ is  locally connected, hence by Lemma \ref{lem:basic nodal} we have that $\Omega$ is open and that either $f>0$ or $f<0$ in $\Omega.$ Assume without loss of generality that $f$ is positive in $\Omega$.
		Define $\psi_n\coloneqq\left(f-\frac{1}{n}\right)^{+}\chi_{\Omega}$ and note that $\sfd(\supp(\psi_n),\X\setminus\Omega)>0.$ We claim that $\psi_n\in W^{1,2}_0(\Omega) $. To see this
		let $\varphi\in \LIP(\rr)$ satisfy $|\phi|\le 1,$ $\varphi(t)=0$  for $t\le 0$ and $\varphi(t)=1$ for $t\geq \sfd(\supp(\psi_n),\Omega^c)$. Then $\psi_n=\left(f-\frac{1}{n}\right)^{+}\varphi(\sfd(\cdot,\X\setminus\Omega))$ with $\left(f-\frac{1}{n}\right)^{+}$ in $W^{1,2}(X)$   from the chain rule (see \eqref{eq:loc and leib}) and $\varphi(\sfd(\cdot,\X\setminus\Omega))\in \LIP\cap L^\infty(\X)$. Therefore the Leibniz rule for the minimal weak upper gradient (see \eqref{eq:loc and leib}) implies that $\psi_n\in W^{1,2}(X)$. Since $\sfd(\supp(\psi_n),\X\setminus\Omega)>0,$  the claim follows (recall \eqref{eq:support inside}). From \eqref{char: lambda1} we have
		\begin{equation}\label{eq:}
		\lambda_1(\Omega)\leq \liminf_{n\to +\infty}\frac{\int_{\Omega}|\nabla \psi_n|^2\, \d \mm}{\int_{\Omega} \psi_n^2\,\d\mm}.
		\end{equation}
		Now we observe that $(\psi_n)_{n\in\mathbb{N}}$ converges to $f\chi_{\Omega}$ in $L^2(\mm)$. Indeed 
		\begin{equation*}
		\lim_{n\to +\infty}\int_{\Omega}|\psi_n-f|^2\,\d\mm\leq\lim_{n\to +\infty}\frac{1}{n^2}\mm\left(\Omega\right)=0. 
		\end{equation*}
		Moreover 
		\begin{align*}
		 \lim_{n\to +\infty}\||D\psi_n|\|^2_{L^2(X)}=\lim_{n\to +\infty}\int_{\Omega \cap \left\{ f> \frac{1}{n}\right\} }|Df|^2\,\d\mm=  \int_{\Omega}|Df|^2\,\d\mm,
		\end{align*}
		where the first equality follows from the locality of the weak upper gradient and the second by monotone convergence theorem. This implies that $f\chi_{\Omega}\in \W(\X)$ (see  e.g.\ Proposition 2.1.19 in \cite{GP20}). Again by the locality of the weak upper  gradient we have  
		\begin{equation*}
		\||D(f\chi_{\Omega})|\|_{L^2(X)}= \||Df|\|_{L^2(\Omega)}, 
		\end{equation*}
		and so $\{\psi_n\}_{n\in \mathbb{N}}$ converges to $f\chi_{\Omega}$ in $W^{1,2}(X)$. 
		In particular 
		\begin{equation*}
		\liminf_{n\to +\infty}\frac{\int_{\Omega}|\nabla \psi_n|^2\, \d \mm}{\int_{\Omega} \psi_n^2\,\d\mm}=\frac{\int_{\Omega} |\nabla f|^2\, \d \mm}{\int_{\Omega} f^2\,\d\mm}.
		\end{equation*}
		The result follows once we observe that 
		\begin{equation}\label{eq:green}
		\frac{\int_{\Omega} |\nabla f|^2\, \d \mm}{\int_{\Omega} f^2\,\d\mm}=\lambda.
		\end{equation}
		To see this   by definition of eigenfunction (both in the Dirichlet case and the Neumann case) and since $\supp(\psi_n)\subset \Omega$ one has
		\begin{equation*}
		-\int_\Omega \nabla f\cdot \nabla \psi_n\,\d\mm=\lambda\int_\Omega f \psi_n\,\d\mm\,, \quad \forall n\in \nn,
		\end{equation*}
		from which \eqref{eq:green} follows passing to the limit and noting
		\begin{equation*}
		\limsup_{n\to +\infty}\left|\int_{\Omega}\nabla f\cdot \nabla \psi_n\,\d\mm-\int_{\Omega}\nabla f\cdot \nabla f\,\d\mm\right|\leq\limsup_{n\to +\infty} \||Df|\|_{L^2(\Omega)}\||D(\psi_n-f\chi_{\Omega})|\|_{L^2(\Omega)}=0,
		\end{equation*}
  where we used the bilinearity of the scalar product, \eqref{eq:C-S} and that $|D(\psi_n-f)|=|D(\psi_n-\nchi_\Omega f)|$ $\mea$-a.e.\ in $\Omega$, by the locality.
	\end{proof}

Before proving Theorem \ref{main th} we give a precise definition of the nodal domain counting function which appears in its statement. 
\begin{definition}[Nodal domains counting functions]\label{def:counting}
    Let $\Xdm$ be an infinitesimally Hilbertian PI space and $\Omega\subset \X$ be a uniform domain (resp.\ bounded open set). The \emph{nodal domain counting function} $M_\Omega^\cN:\nn \to \nn\cup\{+\infty\} $  (resp.\  $M_\Omega^\cD:\nn \to \nn\cup\{+\infty\}$) is given by
    \[
    \begin{split}
        &M_\Omega^\cN(k):=\sup\{M(u) : u \ \textrm{Neumann eigenfunction in $\Omega$ of eigenvalue} \ \lambda_k^\cN(\Omega)\},\\
        & \big (\text{resp.\ }  M_\Omega^\cD(k):=\sup\{M(u) : u \ \textrm{Dirichlet eigenfunction in $\Omega$ of eigenvalue} \ \lambda_k^\cD(\Omega)\}\big),\quad k \in \nn,
    \end{split}
    \]
    where $\{\lambda_k^\cN(\Omega)\}_k$ (resp.\ $\{\lambda_k^\cD(\Omega)\}_k$)  are the eigenvalues of the Neumann (resp.\ Dirichlet) Laplacian in  $\Omega$, which has discrete spectrum by Corollary \ref{cor:spectrum compatilbity} (resp.\ by observation \eqref{eq:PI discrete sepctrum}), and
    where $M(u)$ is the number of nodal domains  of the continuous representative of $u$ in $\Omega$, which exists by Theorem \ref{thm:continuous eigen}.
\end{definition}

We are finally  ready to prove the main result of the note. We will first prove the Neumann case and then the Dirichlet case. The proofs are essentially identical with our approach, but to avoid confusion we decided to keep them separated.
\begin{proof}[Proof of Theorem \ref{main th} in the Neumann case]
Consider  the m.m.s.\ $(\overline \Omega,\sfd\restr {\overline \Omega }, \mmomega)$ and observe that by Theorem \ref{th:almost-iso} it satisfies the hypotheses of Theorem \ref{thm:faber} (see also Remark \ref{rmk:assumptions}). 
 Fix $\eps\in (0,1)$, $\eta>0$  and let $\delta=\delta(   \Omega, \eps,N,\eta)>0,$ $\delta'=\delta'(\Omega,\eps,N)>0$ and $U_{\eps,\eta}\subset \overline{\Omega}$ be the constants and the set given  by Theorem \ref{thm:faber} applied to $(\overline \Omega,\sfd\restr {\overline \Omega }, \mmomega)$.  Recall that $\mea(U_{\eps,\eta})\le \eta.$
 
 Let $u\in{W^{1,2}(\Omega)}$ be a Neumann eigenfunction  in $\Omega$  of eigenvalue $\lambda_k^\cN(\Omega)>0$. Let $\{\Omega_i\}_{i=1}^{M(u)}$ an enumeration of the  nodal domains of its continuous representative. Note that the nodal domains are countable because they are open, since $(\X,\sfd)$ is locally connected (recall Lemma \ref{lem:basic nodal}). However a priori it could be that $M(u)=+\infty.$
Thanks to Corollary \ref{cor:spectrum compatilbity}, we know that there exist $\tilde u\in W^{1,2}(\overline \Omega,\sfd\restr {\overline \Omega }, \mmomega)$ eigenfunction of the Laplacian in $(\overline \Omega,\sfd\restr {\overline \Omega }, \mmomega)$ of the same eigenvalue  $\lambda_k^\cN(\Omega)$ and such that $\tilde u\restr\Omega=u.$ By Theorem \ref{thm:continuous eigen} it holds that $\tilde u$ has a continuous representative in $\overline \Omega.$
Hence thanks to Proposition \ref{prop:components compatibility} we deduce that $\tilde u$ has the same number of nodal domains  $\{\tilde \Omega_i\}_{i=1, \dots, M( u)}$, where $\tilde \Omega_i \subset \overline \Omega$. 
 Thanks to Theorem \ref{th:almost-iso} we can now apply Proposition \ref{prop:Faber Kran PI} and Proposition \ref{prop: neum bound dir} to the m.m.s.\  $(\overline \Omega,\sfd\restr {\overline \Omega }, \mmomega)$ and deduce that for any $i\in{1,\dots, M( u)}$ it holds 
 \begin{equation}\label{eq:final ineq}
      \lambda_k^\cN(\Omega)\geq \lambda_1(\tilde \Omega_i),\quad \lambda_1(\tilde \Omega_i)\geq \frac{C( \Omega)}{\mm(\tilde\Omega_i)^{\frac{2}{N}}},
 \end{equation}
where $\lambda_1(\tilde \Omega_i)$  is the first eigenvalue of the Dirichlet Laplacian computed in the m.m.s.\ $(\overline \Omega,\sfd\restr {\overline \Omega }, \mmomega)$ (see Definition \ref{def: lambda1}) and  where $C(\Omega)>0$ is a  constant depending only on $\Omega$ and $N.$
Combining the two inequalities above we deduce that $M(u)<+\infty$ and that
\begin{align*}
M(u)\left(\frac{C(\Omega)}{\lambda_k^\cN(\Omega)}\right)^{\frac N2}\le \sum_{i=1}^{M( u)} \left(\frac{ C(\Omega)}{\lambda_1(\tilde \Omega_i)}\right)^{\frac N2}\le \sum_{i=1}^{M( u)}\mm(\tilde\Omega_i)\le \mm(\overline \Omega),
\end{align*}
because the sets $\tilde \Omega_i$ are pairwise disjoint.
We define three sets:
\[
\cal S_1\coloneqq\{i\in 1, ...,M(u) \ : \  \mm(\tilde\Omega_i)\ge \delta \}\, ,
\]
\[
\cal S_2\coloneqq\{i\in 1, ...,M(u) \ : \  \mea(\tilde \Omega_i\cap U_{\eps,\eta})\ge \delta' \mea(\tilde \Omega_i)\}\, ,
\]
\[
\cal S_3\coloneqq\{i\in 1, ...,M(u) \ : \   \mea(\tilde\Omega_i)\le \delta , \mea(\tilde\Omega_i\cap U_{\eps,\eta})\le \delta' \mea(\tilde\Omega_i) \}\, .
\]
Notice that $\cal S_1 \cup \cal S_2 \cup \cal S_3=\{1,...,M(u)\}$.  Clearly
\begin{equation}\label{eq:card1}
	\#\cal S_1 \le \frac{\mm(\overline{\Omega})}{\delta}.
\end{equation}
On the other hand, similarly as above, using \eqref{eq:final ineq}  
\[
\#\cal S_2 \cdot  \delta'\left(\frac{ C(\Omega)}{\lambda_k^\cN(\Omega)}\right)^{N/2} \le \sum_{i \in \cal S_2} \delta' \mm(\tilde\Omega_i)\le  \sum_{i \in \cal S_2}\mea(\tilde\Omega_i\cap U_{\eps,\eta})\le \mm(U_{\eps,\eta})\le \eta\, ,
\]
and thus 
\begin{equation}\label{eq:card2}
	\#\cal S_2 \le  \left( \frac{ \lambda_k^\cN(\Omega)}{C({\Omega})}\right)^{\frac{N}{2}}\frac{\eta}{\delta'}.
\end{equation} 
Finally by Theorem \ref{thm:faber} it holds that $\lambda_1(\tilde \Omega_i)\ge \lambda_1(\Omega_i^*)(1-\eps)$ for every $i \in \cal S_3$, where  $\Omega_i^*\coloneqq B_r(0)\subset \rr^N$ is the ball satisfying $\mm(\tilde \Omega_i)=\Leb^N(\Omega_i^*)$. Therefore for every $i \in \cal S_3$
\begin{equation*}
	\lambda_k^\cN(\Omega)\overset{\eqref{eq:final ineq}}{\ge}\lambda_1(\tilde \Omega_i)\ge \lambda_1( \Omega_i^*)(1-\eps)\overset{\eqref{eq:lambda eucl ball}}{=}(1-\eps) \left(\frac{\omega_N}{\Leb^N(\Omega_i^*)}\right)^{\frac 2N}j^2_{\frac{(N-2)}{N}}=(1-\eps) \frac{\alpha_N }{\mea(\tilde \Omega_i)^{\frac 2N}},
\end{equation*}
where we have put $\alpha_N:=\omega_N^{2/N}j^2_{\frac{(N-2)}{N}}$. This leads to
\begin{equation}\label{eq:card3}
	\#\cal S_3 \le \left(\frac{\lambda_k^\cN(\Omega)}{(1-\eps) \alpha_N}\right)^{\frac N2}\mm(\overline{\Omega}).
\end{equation}
 Combining \eqref{eq:card1}, \eqref{eq:card2} and \eqref{eq:card3} we reach
\[
\frac{M( u)}{k}\le \frac{\#\cal S_1+\#\cal S_2+\#\cal S_3}{k}\leq \frac{\mm(\overline{\Omega})}{k\delta(   \Omega, \eps,N,\eta)}+  \frac{ \lambda_k^\cN(\Omega)^{\frac{N}{2}}}{k}\frac{1}{C({\Omega})^{\frac{N}{2}}}\frac{\eta}{\delta'(\Omega,\eps,N)}
+ \left(\frac{\lambda_k^\cN(\Omega)}{(1-\eps) \alpha_N}\right)^{\frac N2}\frac{\mm(\overline{\Omega})}{k}.
\]
Note that the right hand side is independent of $u$, hence we can take the supremum among all Neumann eigenfunctions in $\Omega$ of eigenvalue $\lambda_k^\cN(\Omega)$ and obtain
\begin{equation*}
    \frac{M_\Omega^\cN(k)}{k}\le \frac{\mm(\overline{\Omega})}{k}\frac{1}{\delta(   \Omega, \eps,N,\eta)}+ 
\frac{ \lambda_k^\cN(\Omega)^{\frac{N}{2}}}{k}\frac{1}{C({\Omega})^{\frac{N}{2}}}\frac{\eta}{\delta'(\Omega,\eps,N)}+ \left(\frac{\lambda_k^\cN(\Omega)}{(1-\eps) \alpha_N}\right)^{\frac N2}\frac{\mm(\overline{\Omega})}{k},
\end{equation*}
where $M_\Omega^{\cN}(k)$ is as in Definition \ref{def:counting}.
Passing to the $\limsup$ as $k \to +\infty$, using \eqref{eq:weyl},

\begin{equation}\label{eq:fine}
\limsup_{k \to +\infty }\frac{M_\Omega^{\cN}(k)}{k}\le  \frac{(2\pi)^N}{\omega_N\mm(\overline \Omega)}\frac{1}{C(\overline{\Omega})^{\frac{N}{2}}}\frac{\eta}{\delta'(\Omega,\eps,N)}+ \frac{(2\pi)^N}{\omega_N} \frac{1}{\left((1-\eps) \alpha_N\right)^{\frac N2}},
\end{equation}
sending first $\eta \to 0$ and then $\eps \to 0$ (note that both $\delta'$ and $\eps$ are independent of $\eta$) we conclude that 
\begin{equation*}
\limsup_{k \to +\infty }\frac{M_\Omega^{\cN}(k)}{k}\le  \frac{(2\pi)^N}{\omega_N\alpha_N^{\frac N2}}=\frac{(2\pi)^N}{\omega^2_N j^N_{\frac{(N-2)}{N}}}<1,
\end{equation*}
where for the last inequality we refer to \cite[Lemme 9]{BM82}.
\end{proof}

\begin{proof}[Proof of Theorem \ref{main th} in the Dirichlet case]
 Let $u\in{W^{1,2}_0(\Omega)}$ be a Dirichlet eigenfunction  in $\Omega$  of eigenvalue $\lambda_k^\cD(\Omega)$.
Without loss of generality we can assume that $\Xdm$ satisfies the hypotheses of Theorem \ref{thm:faber}. Otherwise we can take a uniform domain $U\subset \X$ such that $\Omega \subset U$ (which exists by \cite{Raj20}) and replace $\Xdm$ with $(\overline U,\sfd\restr {\overline U },\mm\restr U)$, which by Theorem \ref{th:almost-iso} satisfies the hypotheses of Theorem \ref{thm:faber}. Moreover it is a direct verification  that the function $u\in L^2(\Omega)$  remains a Dirichlet eigenfunction  in $\Omega$  of eigenvalue $\lambda_k^\cD(\Omega)$ also in the new space $(\overline U,\sfd\restr {\overline U },\mm\restr U)$. Indeed  $\W_0(\Omega)$ viewed as a subset of $L^2(\Omega)$ in the space $\Xdm$ coincides with $\W_0(\Omega)$ viewed as a subset of $L^2(\Omega)$ in the space $(\overline U,\sfd\restr {\overline U },\mm\restr U)$, and the corresponding minimal w.u.g.\ also coincide (see \cite[Prop.\ 6.4]{AMS16}).

 Fix now $\eps\in (0,1)$, $\eta>0$  and let $\delta=\delta(   \X, \eps,N,\eta)>0,$ $\delta'=\delta'(\X,\eps,N)>0$ and $U_{\eps,\eta}\subset \X$ be the constants and the set given  by Theorem \ref{thm:faber} applied to $\Xdm$.  Recall that $\mea(U_{\eps,\eta})\le \eta.$
From here the proof proceeds almost verbatim as in the Neumann case, by considering the sets
\[
\cal S_1\coloneqq\{i\in 1, ...,M(u) \ : \  \mm(\Omega_i)\ge \delta \}\, ,
\]
\[
\cal S_2\coloneqq\{i\in 1, ...,M(u) \ : \  \mea( \Omega_i\cap U_{\eps,\eta})\ge \delta' \mea( \Omega_i)\}\, ,
\]
\[
\cal S_3\coloneqq\{i\in 1, ...,M(u) \ : \   \mea(\Omega_i)\le \delta , \mea(\Omega_i\cap U_{\eps,\eta})\le \delta' \mea(\Omega_i) \},
\]
where $\{\Omega_i\}_{i=1}^{M(u)}$ is an enumeration of the  nodal domains of the continuous representative of $u$, and then exploiting the inequalities
 \begin{equation*}
      \lambda_k^\cN(\Omega)\geq \lambda_1( \Omega_i),\quad \lambda_1(\Omega_i)\geq \frac{C_\X}{\mm(\Omega_i)^{\frac{2}{N}}}, \quad \text{ for all  $i\in{1,\dots, M( u)}$}
 \end{equation*}
which hold by Proposition \ref{prop:Faber Kran PI} and Proposition \ref{prop: neum bound dir} (that we can apply again by Theorem \ref{th:almost-iso}), together with
 \[
 \lambda_1( \Omega_i)\ge \lambda_1(\Omega_i^*)(1-\eps), \quad \text{for all $i \in \cal S_3$},
 \]
 ($\Omega_i^*\coloneqq B_r(0)\subset \rr^N$ being the ball satisfying $\mm( \Omega_i)=\Leb^N(\Omega_i^*)$)  that holds by Theorem \ref{thm:faber}.
\end{proof}

For completeness, we conclude with the proofs of the Corollaries \ref{cor: mainRCD} and \ref{cor: mainEucl}, even if they are essentially already included in Theorem \ref{main th}.
\begin{proof}[Proof of Corollary \ref{cor: mainRCD}]
We apply Theorem \ref{main th} with $\Omega=\X$. Notice that the choice is admissible since $\X$ is assumed to be compact, and thus $\Omega$ is trivially a uniform domain inside $\X$. From the discussion in Section \ref{sec:calculus} we know that $\Delta=\Delta_\cN$ in this situation, and thus the result follows. 
\end{proof}
\begin{proof}[Proof of Corollary \ref{cor: mainEucl}]
    The metric measure space $(\rr^N,|\cdot|,\hau^N)$, where $\hau^N$ is the $N$-dimensional Hausdorff measure is a non-collapsed $\RCD(0,N)$ (recall Remark \ref{rmk:rcd ok}). Hence the result follows applying Theorem \ref{main th}, recalling also the compatibility between Neumann eigenfunctions in the metric setting and the usual ones (see Remark \ref{rmk:eucl lapl}).
\end{proof}

  \bibliographystyle{siam}
\bibliography{courantbib} 

\end{document}